\documentclass[11pt]{amsart}
\usepackage{amsmath}
\usepackage{amsfonts}
\usepackage{amssymb}
\usepackage[normalem]{ulem} 

\usepackage[utf8]{inputenc}
\usepackage[T2A]{fontenc}
\usepackage[russian,english]{babel}
\inputencoding{utf8}

\newtheorem{theorem}{Theorem}[section]

\newtheorem{corollary}[theorem]{Corollary}
\newtheorem{proposition}[theorem]{Proposition}
\theoremstyle{definition}
\newtheorem{definition}[theorem]{Definition}
\newtheorem{example}[theorem]{Example}

\theoremstyle{remark}
\newtheorem{remark}[theorem]{Remark}
\numberwithin{equation}{section}
\numberwithin{equation}{section}

\title[Modular differential equations of minimal order]
{Modular differential equations of minimal orders of the elliptic genus of Calabi--Yau varieties}
\author{Dmitrii Adler}
\address{Max Planck Institute for Mathematics, Vivatsgasse 7, 53111 Bonn, Germany}
\email{dmitry.v.adler@gmail.com}

\author{Valery Gritsenko}

\address{International laboratory of mirror symmetry and automorphic forms, HSE University, Moscow and Laboratoire Paul Painlev\'e, Universit\'e de Lille}

\email{valery.gritsenko@univ-lille.fr}
\date{\today}
\begin{document}
\maketitle

\begin{abstract}
We study modular differential equations (MDEs) of high orders and find necessary conditions for weak Jacobi forms  to satisfy MDEs of order 3 with respect to the heat operator. 
We investigate all possible MDEs for weak Jacobi forms of weight 0 and index 3.
This is the target space for the elliptic genus of the compact complex  manifolds  of dimension 6 with trivial  first Chern class. We prove that the minimal possible order of MDEs of such Jacobi forms is four. Moreover, we find all such forms and show that only three of them  might be the elliptic genus of strict Calabi--Yau six-folds. We describe also a discrete set of  Jacobi forms satisfying fifth-order MDEs  and the divisor of forms  satisfying sixth-order MDEs. Then we prove that a Jacobi forms of weight 0 and index 3 which does not belong to a smooth cubic in the space of coefficients satisfies a MDE of order 7. We provide such MDEs for the elliptic genus of 6-dimensional holomorphic symplectic varieties of types $\hbox{Hilb}^{[3]}(K3)$, $\hbox{Kum}_3(A)$, and $\hbox{OG}6$.
\end{abstract}

\section{Introduction}
This paper is a continuation of the research started in  
\cite{AG1} and \cite{AG2}, where modular differential equations (MDEs) for weak Jacobi forms were studied.  In contrast to MDEs for modular forms in one variable, in similar equations for Jacobi forms, the differentiation with respect to the modular parameter $q=e^{2\pi i \tau}$ is replaced by the heat operator (see \eqref{MDEdef} for the definition of MDEs). The elliptic genus of a compact complex  manifold of dimension $d$ with vanishing first Chern class is a weak Jacobi form of weight 0 and index $\frac{d}{2}$ with integer Fourier coefficients (see \cite{KYY}, \cite{Gr99}). 
 
In \cite{AG1}, we proved that the elliptic genus of Calabi--Yau varieties of dimension 3 satisfies the simplest MDE of order one with respect to the heat operator. The elliptic genus of a $K3$ surface, as well as of a Calabi--Yau variety of dimensions 5, satisfies a third-order MDE. 
We also found examples of solutions to second-order modular differential equations, analogous to the Kaneko--Zagier equation in the theory of modular forms (see \cite{KZ} and \cite{KK}). 
In  addition, MDEs of the four canonical generators of the ring of weak Jacobi forms of weight 0 with integer coefficients were calculated. 
It is interesting because these Jacobi forms  are partition functions of Lorentzian Kac--Moody algebras with a hyperbolic root systems of signature $(2,1)$. 

In \cite{AG2}, we described solutions of MDEs of order 1 and  examined Kaneko--Zagier type MDEs in more detail. In particular, we proved that the elliptic genus of a Calabi--Yau variety (in the strict sense) does not satisfy a second-order MDE unless its dimension is 3. We define a strict Calabi-Yau variety as a simply connected smooth projective variety $M$ with a trivial canonical class and $h^{p,0}(M)=0$ for every $0<p<\hbox{dim}(M)$.
In \cite{AG2}, we classified all possible MDEs for Jacobi forms of weight 0 and index 2. As an application, we identified two types of strict Calabi--Yau fourfolds with Euler number 48 and $-18$. Their elliptic genera satisfy a MDE of the minimal possible order --- namely, order 3.

The case of index 3 is very different from those studied in \cite{AG1}--\cite{AG2}. We denote the three-dimensional space of the weak Jacobi forms of weight 0 and index 3 by $J_{0,3}^w$ (see all necessary definitions in \S 2).
The method of finding MDEs, which we proposed in \cite{AG1}, combined with linear algebra arguments, allowed us to show that a generic Jacobi form in $J_{0,3}^w$ satisfies a MDE of order $7$ (see \cite[Theorem 4.1]{AG1}).  (Here, ``generic'' denotes Jacobi forms outside of a lower-dimensional locus.) In this paper,  we present a detailed study of all possible  MDEs for Jacobi forms of weight 0 and index 3.
 
In \S 3, we analyse third-order  MDEs of Jacobi forms of any integral weight and index. We show that only very special Jacobi forms of weight 0 satisfy third-order MDEs (see Theorems \ref{MDE3k} and \ref{q0-Ord3}). Consequently, the occurrence of such MDEs for the elliptic genus of strict Calabi--Yau varieties is quite rare (see Corollary \ref{EG-Ord3}). All known examples were found in \cite{AG1} and \cite{AG2} and are mentioned above. 

In Theorem \ref{ind3wt0}, we prove that the minimum possible order for the MDEs satisfied by weak Jacobi forms of weight 0 and index 3 is 4.  We then demonstrate that there are exactly ten such Jacobi forms (up to a scalar factor) that satisfy MDEs of order 4 (Theorem \ref{Thm-MDE4}). Of these ten, only three may correspond to the elliptic genera of strict six-dimensional Calabi--Yau varieties  (see Corollary \ref{Cor-CYMDE4}).  Furthermore, we show that there exist only five weak Jacobi forms of weight 0 and index 3 that satisfy fifth-order MDEs but do not satisfy any MDE of lower order (see Theorem \ref{ThmMDE5}). These MDEs provide the discrete part of our classification for special MDE solutions in $J_{0,3}^w$.

The study of higher-order equations introduces a new phenomenon.
We find five one-parameter families (four projective lines and one 
smooth cubic curve $Q$ in the projective plane $\Bbb P^2(\Bbb C)$ of 
coefficients) for which Jacobi forms satisfy MDEs of order 6 (see Theorem 
\ref{TMDE6}). Any two of these five divisors have a common point that corresponds to 
a Jacobi form satisfying a fourth-order  MDE. Furthermore, every Jacobi form from Theorem  \ref{ThmMDE5} satisfying a fifth-order MDE belongs to exactly one of these five divisors.

We further prove (see Theorem \ref{TMDE7}) that all Jacobi forms not corresponding to points of another smooth cubic curve $S$ in $\Bbb P^2(\Bbb C)$ satisfy a  MDE of  order 7. This class includes, for example, the elliptic genera of the 6-dimensional hyperk\"ahler varieties of types $\hbox{Hilb}^{[3]}(K3)$, $\hbox{Kum}_3(A)$, and $\hbox{OG}6$. The corresponding MDEs are provided in Example \ref{KAOG}.

However, Jacobi forms that correspond to points on the plane cubic $S$ but do not lie on any of the five divisors describing sixth-order MDEs do not satisfy an MDE of order 7. For this exceptional subset, one must consider MDEs of order 8 or higher. While this investigation could be conducted using the proposed method, it was beyond the scope of the present paper.

In the conclusion of the paper, we demonstrate that the order of MDEs can be reduced if we allow them to have meromorphic modular coefficients.

\section{Jacobi modular forms, elliptic genus and MDEs}
This section presents the necessary definitions and preliminary results. In particular, we provide examples of modular differential equations (MDEs) for the elliptic genus of Calabi--Yau varieties of dimensions $2, 3,$ and $5$ from \cite{AG1}--\cite{AG2}.

\subsection{Definition of Jacobi forms}\label{Sec:2.1}
The natural model for the Jacobi modular group $\Gamma^J(\mathbb{Z})$ is the quotient $\Gamma_{\infty}/\{\pm I\}$, where $\Gamma_{\infty}$ is the integral maximal parabolic subgroup of the Siegel modular group of genus $2$ that stabilizes an isotropic line. For this viewpoint and the definitions that follow, we refer the reader to \cite[\S 1]{GN98}.

The Jacobi group $\Gamma^J(\mathbb{Z})$ is the semidirect product of $SL_2(\mathbb{Z})$ and the integral Heisenberg group $H(\mathbb{Z})$. The group $H(\mathbb{Z})$ is a central extension of $\mathbb{Z}\times\mathbb{Z}$ and the unipotent subgroup of $\Gamma_\infty(\mathbb{Z})$.

\begin{definition}\label{def:JF}
Let $k\in \Bbb Z$ and $m\in \frac{1}{2}\Bbb N$. 
A holomorphic function
\newline 
$\varphi : \Bbb H\times \Bbb C\to \Bbb C$ is a weak Jacobi 
form of weight $k$ and index $m$  if it satisfies  two functional equations
\begin{align*}
\varphi\left(\frac{a\tau+b}{c\tau+d},\frac{z}{c\tau+d}\right)&
=(c\tau+d)^ke^{2\pi i m \frac{cz^2}{c\tau+d}}\,\varphi(\tau,z),\\
\varphi(\tau,z+x\tau+y)&=(-1)^{2m(x+y)}\,e^{-2\pi i m (x^2\tau+2xz)}\varphi(\tau,z)
\end{align*}
for any $\left(\begin{smallmatrix}
a & b \\ 
c & d \end{smallmatrix}\right)\in SL_2(\Bbb Z)$ and
$x,y\in \Bbb Z$,
and has a Fourier expansion of the form
\begin{equation}\label{F-exp} 
\varphi(\tau,z)=\sum_{\substack{ n\geq 0,\ l\in \frac{1}{2}\Bbb Z}} a(n,l)\exp(2\pi i(n\tau+lz)) 
\end{equation}
where $l$ is a half-integral for a half-integral index $m$. 
If $\varphi$ satisfies the additional condition $a(n,l)=0$ for $4nm - l^2<0$, it is called a {\it holomorphic} (at infinity) Jacobi form. 
\end{definition}

We denote by $J_{k,m}^w$ the finite dimensional vector space of weak Jacobi forms of weight $k$ and index $m$. When $m=0$, the space
$J_{k,0}^w=M_k(SL_2(\Bbb Z))$ coincides with the space of modular forms 
for the full modular group.

We use the Jacobi triple product formula for the odd Jacobi theta-function of characteristic two 
\begin{equation*}\label{theta}
\begin{split}
\vartheta(\tau ,z)&=
q^{1/8}(\zeta^{1/2}-\zeta^{-1/2})\prod_{n\ge 1}\,(1-q^{n}\zeta)(1-q^n 
\zeta^{-1})(1-q^n)\\
{}&=q^{\frac{1}{8}}\zeta^{\frac{1}{2}}
\sum_{n\in \Bbb Z}\,(-1)^nq^{\frac{n(n+1)}{2}}\zeta^{n}, \quad
q=e^{2\pi i \tau},\ \zeta=e^{2\pi i z}.
\end{split}
\end{equation*}
The function $\vartheta(\tau,z)$ satisfies the relations $\vartheta(\tau, -z)=-\vartheta(\tau, z)$ and $\vartheta(\tau, z)=-i\vartheta_{11}(z,\tau)$ in the notation of \cite[Chapter 1]{M}.  We have 
\begin{equation*}\label{G2}
\frac{\partial\vartheta(\tau ,z)}{\partial z}\big|_{z=0}
=2\pi i \,\eta(\tau )^3,\quad
\dfrac{\partial^2 \log \vartheta(\tau,z)}{\partial z^2}=
-\wp(\tau,z)-\frac{\pi^2}3 E_2(\tau),
\end{equation*}
where $\wp(\tau,z)$ is the Weierstrass $\wp$-function, 
$\eta(\tau)\in M_{\frac{1}2}(SL_2(\Bbb Z),v_\eta)$ is the Dedekind eta-function ($v_\eta^{24}=1$), and 
\begin{equation}\label{E2} 
E_2(\tau)=1-24\sum_{n\ge 1 }\sigma_1(n)q^n
\end{equation}
is the quasi-modular Eisenstein series of weight $2$.

The Jacobi theta-series $\vartheta(\tau ,z)$ is a holomorphic Jacobi form of weight $\frac{1}2$ and index 
$\frac{1}2$ in $J_{\frac{1}2, \frac{1}2}(v_\eta^3\times v_H)$. Using 
$\vartheta(\tau ,z)$ we obtain two fundamental weak Jacobi forms
\begin{equation}\label{phi-21}
\varphi_{-2,1}=\frac{\vartheta(\tau,z)^2}{\eta(\tau)^6}
=\zeta^{\pm 1}-2+ 
q(-2\zeta^{\pm 2} +8\zeta^{\pm 1}-12)+O(q^2),
\end{equation}
\begin{equation}\label{phi01}
\varphi_{0,1}=-\frac{3}{\pi^2}\wp(\tau,z)\varphi_{-2,1}=
\zeta^{\pm 1}+10+ q(10\zeta^{\pm 2} - 64\zeta^{\pm 1}+108)+O(q^2)
\end{equation}
where $\zeta^{\pm l}= \zeta^{l}+\zeta^{-l}$.
According to \cite[Theorem 9.3]{EZ} the bigraded ring of weak 
Jacobi forms of even weight and integral index is a polynomial ring 
with four generators 
\begin{equation*}\label{J2**}
J_{2*,*}^w=\bigoplus_{k\in \Bbb Z,\, m\in \Bbb Z_\ge 0}  
J_{2k,m}^w=
\Bbb C[E_4, E_6, \varphi_{-2,1},\varphi_{0,1}]
\end{equation*}
where 
$E_4(\tau)=1+240\sum_{n\ge 1}\sigma_3(n)q^n$ and 
$E_6(\tau)=1-504\sum_{n\ge 1}\sigma_5(n)q^n$
are the  Eisenstein series of weight $4$ and $6$. 
For weak Jacobi forms of odd weights or half-integral indices, the following relations hold (see \cite{GN98}, \cite{Gr99}):
$$
J_{2k+1,m}^w=\varphi_{-1,2}\cdot J_{2k+2,m-2}^w,\quad
J_{2k+1,m+\frac{1}{2}}^w=\varphi_{-1,\frac{1}{2}}\cdot J_{2k+2,m}^w,
$$
$$
J_{2k,m+\frac{1}{2}}^w=\varphi_{0,\frac{3}{2}}\cdot J_{2k,m-1}^w
$$ 
for any 
$k\in \Bbb Z$, $m\in \Bbb N$,
where 
\begin{equation}\label{phi32}
 \varphi_{-1,\frac{1}{2}}(\tau,z)=\frac{\vartheta(\tau,z)}
{\eta^3(\tau)}, \ \varphi_{-1,2}(\tau,z)=
\frac{\vartheta(\tau,2z)}{\eta^3(\tau)},\ 
\varphi_{0,\frac{3}{2}}(\tau,z)=\frac{\vartheta(\tau,2z)}
{\vartheta(\tau,z)}.
\end{equation}

\subsection{Elliptic genus of compact complex manifolds with $c_1=0$.}
\label{Sec:2.2}
Definitions of the elliptic genus in two variables for Calabi--Yau varieties can be found in various contexts; see \cite{KYY}, \cite{Gr99}, and \cite{T00}. More generally, a modified elliptic genus can be defined as a meromorphic partition function for a vector bundle over a complex manifold (see \cite{Gr20}).

Let $M=M_d$ be an (almost) complex compact manifold of (complex) dimension $d$, with tangent bundle $T_M$. The definition of $\chi(M_d;\tau,z)$ is a geometric realization of the Jacobi triple product formula for the Jacobi theta-series introduced above. Let $\tau\in \mathbb{H}$ be a variable in the upper half-plane and $z\in \mathbb{C}$. Setting $q=\exp(2\pi i \tau)$ and $\zeta=\exp(2\pi i z)$, we define the formal series

$$
\Bbb {E}_{q,\zeta}=  \bigotimes_{n= 0}^{\infty}
{\bigwedge}_{-\zeta^{-1}q^{n}}T_M^*
\otimes 
 \bigotimes_{n= 1}^{\infty}{\bigwedge}_{-\zeta q^n}  T_M\otimes 
 \bigotimes_{n= 1}^{\infty} S_{q^n} T_M^*
\otimes 
 \bigotimes_{n= 1}^{\infty} S_{q^n} T_M
$$
where $\wedge^k$ is the $k^\text{th}$ exterior power, $S^k$ is the $k^\text{th}$ symmetric product, and
$$ 
{\bigwedge}_x E=\sum_{k\ge 0}  (\wedge^k E) x^k, \quad S_x E= \sum_{k\ge 0} (S^k E) x^k.
$$
The elliptic genus of a manifold $M$ is defined as the following function of two variables $\tau\in \mathbb{H}$ and $z\in \mathbb{C}$ \cite{KYY, Gr99}:
$$
\chi(M_d; \tau,z)= \zeta^{d/2}\int_{M}  
\operatorname{ch}(\Bbb E_{q,\zeta})\operatorname{td}(T_{M})
$$
where  $\operatorname{td}$ is the Todd class, $\operatorname{ch}({\Bbb E}_{q,\zeta})$ is the Chern character applied to each coefficient of the formal power series $\mathbb{E}_{q,\zeta}$, and $\int_M$ denotes the evaluation of the top degree differential form on the fundamental cycle of the manifold.

The elliptic genus has a Fourier expansion of the form
$$
\chi(M_d;\tau,z)=\sum_{n\ge 0,\, l\in \frac 12\Bbb Z}
a(n,l)\,q^n\zeta^l
$$
where the coefficients $a(n,l)\in \mathbb{Z}$ are given by the index of the Dirac operator twisted by the corresponding vector bundle from the formal series $\mathbb{E}_{q,\zeta}$.

The $q^0$-term of the elliptic genus is equal (up to a renormalization) to the Hirzebruch $\chi_y$-genus of $M_d$, namely, the elliptic genus equals
\begin{equation*}\label{chiy}
\chi(M_d;\tau,z)=\sum_{p=0}^{d}(-1)^p\chi^p(M_d)\zeta^{d/2-p}+O(q)=
\zeta^{\frac d2}\chi_{-\zeta^{-1}}(M_d)+O(q)
\end{equation*}
where $\chi^p(M_d)=\sum_{q=0}^d(-1)^qh^{p,q}(M_d)$. 

\begin{theorem}\label{Th:EG}
{\rm (See \cite{Gr99}, \cite{KYY}.)}
 If $M_d$ is a compact 
complex manifold of dimension $d$ with $c_1(M_d) = 0$ (over $\Bbb R$), 
then its elliptic genus $\chi(M_d; \tau, z)$ is a weak Jacobi form of 
weight $0$ and index $\frac{d}2$ with integer Fourier coefficients.
\end{theorem}

\begin{example} 
\begin{equation}\label{EGK3}
\chi(K3; \tau,z)=2\varphi_{0,1}(\tau, z),\quad \chi(CY_3; \tau, z)=\frac{e(CY_3)}2\varphi_{0,\frac32}(\tau, z),
\end{equation}
\begin{equation}\label{EGCY5}
\chi(CY_5; \tau, z)=\frac{e(CY_5)}{24}\varphi_{0,1}(\tau, z)\varphi_{0,\frac32}(\tau, z)
\end{equation}
where $e(M)$ is the Euler number,  and the corresponding Jacobi forms are defined in  \eqref{phi01} and \eqref{phi32}. 
\end{example} 
\subsection{Modular differential operators for modular forms and Jacobi forms}\label{Sec:2.3}
For $q=e^{2\pi i \tau}$ we  put
$
D=q\frac{d\ }{dq}=\frac{1}{2\pi i}\frac{d\ }{d\tau}
$.
The modular differential operator  is defined by
$$
D_k: M_k(SL_2(\Bbb Z))\to M_{k+2}(SL_2(\Bbb Z)), \quad 
D_k(f)=D(f)-\frac{k}{12}E_2\cdot f
$$
where $E_2(\tau)$ is the quasi-modular Eisenstein series \eqref{E2}.
This operator can be extended to a modular differential operator on weak Jacobi forms, denoted by
$$
H_k:  J_{k,m}^w\to  J_{k+2,m}^w
$$
which is defined by
$$
H_{k}(\varphi_{k,m})=
\left(12q\frac {d\ }{dq}- \frac{3}m\left(\zeta\frac {d\ }{d\zeta}\right)^2\right)(\varphi_{k,m}) -\frac{2k-1}{2}E_2(\tau)\varphi_{k,m}
$$
where  $\zeta=e^{2\pi i z}$ and 
\begin{equation}\label{Heat}
H=H^{(m)}=12q\frac {d\ }{dq}- \frac{3}m\left(\zeta\frac {d\ }{d\zeta}\right)^2
\end{equation}
is the renormalized {\it heat operator}. This choice of normalization is made to simplify calculations. The heat operator act explicitly on the Fourier expansion of Jacobi forms by
\begin{equation*}\label{FC-heat}
\left(12q\frac {d\ }{dq}- \frac{3}m\left(\zeta\frac {d\ }{d\zeta}\right)^2\right)(q^n\zeta^l)=\frac 3{m}(4nm-l^2)\,q^n\zeta^l
\end{equation*} 
where  $4nm-l^2$ is the hyperbolic norm of the index $(n,l)$ of  the Fourier coefficient $a(n,l)$ of a Jacobi form of index $m$ (see (\ref{F-exp})).  For example, $H_{-2}(\varphi_{-2,1})=-\frac 12\varphi_{0,1}$.
The operators $D_k$ and $H_k$ satisfy the following product rule:
\begin{equation*}\label{Hk-prod}
H_{k_1+k_2}(f_{k_1}\varphi_{k_2,m})=
12D_{k_1}(f_{k_1})\varphi_{k_2,m}
+f_{k_1}H_{k_2}(\varphi_{k_2,m})
\end{equation*}
for $f_{k_1}(\tau)\in M_{k_1}(SL_2(\Bbb Z))$ and $\varphi_{k_2,m}(\tau,z)\in J_{k_2,m}^w$ (see \cite[Proposition 2.2]{AG2}).
In particular, $H_k$ is homogeneous with respect to multiplication by the Dedekind eta-function:
\begin{equation}\label{H-eta}
H_{\frac n2 +k}(\eta^{n}\varphi_{k,m})=\eta^{n}H_{k}(\varphi_{k,m}).
\end{equation}
We denote the $n$-th iterate of the operator $H_{k}$ by
\begin{equation}\label{H-iter}
H_k^{[n]}=  H_{k+2(n-1)}\circ \ldots \circ H_{k+2}\circ H_k, \quad H_k^{[1]}=H_k.
\end{equation}
A  {\it modular differential equation} (or {\it MDE}) of order $n$ is a differential equation of the form
\begin{equation}\label{MDEdef}
H_k^{[n]}(\varphi_{k,m})+\sum_{t=1}^{n-1} f_{2(n-t)}(\tau)H_k^{[t]}(\varphi_{k,m})=0
\end{equation}
where $f_{2(n-t)}\in M_{2(n-t)}(SL_2(\Bbb Z))$ is a modular form of weight 
$2(n-t)$.

\subsection{Examples of MDEs of order 1, 2 and 3.}\label{Sec:2.4}
In this subsection, we consider the case of Calabi--Yau varieties $CY_2$, $CY_3$, and $CY_5$ of dimension $2$, $3$, and $5$, respectively. In these cases the elliptic genus of corresponding Calabi--Yau varieties depends only on the Euler characteristic of the variety. 

{\bf MDEs of order 1.} In \cite{AG1} we showed that the elliptic genus of a Calabi-Yau threefold
$\chi(CY_3,\tau,z)=\frac{e(CY_3)}2\phi_{0, \frac 32}$ 
satisfies the simplest MDE
$$
H_0(\phi_{0, \frac 32})=0.
$$
In \cite[Theorem 2.2]{AG2} we described all weak Jacobi forms satisfying 
the MDEs  $H_k(\varphi_{k,m})=0$ of order one.
\smallskip

{\bf MDEs of order 2.} A general MDEs of order $2$ is of the  form 
\begin{equation*}\label{HHk}
H_{k+2}H_k(\varphi_{k,m})=\lambda E_4\varphi_{k,m}
\end{equation*}
where $E_4(\tau)$ is the Eisenstein series of weight $4$ and $\lambda\in \Bbb C$. This equation generalizes the \textit{Kaneko--Zagier equation} for modular forms, which finds various applications in number theory and the theory of vertex operator algebras (see \cite{KZ}, \cite{KK}, \cite{KNS}, \cite{KS}, and also \cite{K} for the case of holomorphic Jacobi forms of index 1).

It was shown in \cite[Theorem 3.3]{AG2} that for weak Jacobi forms of weight $0$, the equation
$$
H_2H_0(\varphi_{0,m})= \lambda E_4\varphi_{0,m}
$$
admits only trivial solutions of the form $\varphi_{0,m}=c\varphi_{0,\frac{3}{2}}(\tau, az)$ where $a\in \mathbb{Z}$ and $H_0(\varphi_{0,m})=0$. Consequently, the elliptic genus of a Calabi--Yau variety (in a strict sense) does not satisfy a modular differential equation of order $2$, unless its dimension is three.

\smallskip

{\bf MDEs of order 3.}
The elliptic genus of a K3 surface and a Calabi--Yau variety of dimension $5$ (see \eqref{EGK3}--\eqref{EGCY5}) satisfies  MDE of order 3
(see \cite[Theorem 2.3]{AG1})
\begin{equation}\label{mde:K3}
\bigl(H_0^{[3]}-\tfrac{101}4E_4H_0+10E_6\bigr)(\varphi_{0,1})=0,
\end{equation}
\begin{equation}\label{mde:CY5}
\bigl(H_0^{[3]}-\frac{611}{25}E_4H_0+\frac{88}{25}E_6\bigr)(\varphi_{0,1}\cdot\varphi_{0,\frac 52})=0.
\end{equation}
Here we use notation \eqref{H-iter}.
In the next section, we provide a more detailed analysis of MDEs of order 3.

\section{MDEs of order 3}

A general MDE of order 3 has two coefficients and takes the form
\begin{equation}\label{H3k}
(H_k^{[3]}+\lambda_1E_4H_k+\lambda_2E_6)(\varphi_{k,m})=0.
\end{equation}
This equation can be represented in terms of the heat operator $H$ with quasi-modular coefficients \cite[Proposition 5.1]{AG2}. For weight $k=0$, the equation takes the following form, where $H=H^{(m)}$ is the heat operator \eqref{Heat}:
\begin{multline*}
\bigl(H^3-\frac 92E_2H^2+ 27E'_2H+\bigl(\frac{11}4+\lambda_1\bigr)E_4H+\\
27E''_2
+\frac 32\bigl(\frac{11}4+\lambda_1\bigr)E_4'+\bigl(\frac{21}8+\frac{\lambda_1}2+\lambda_2\bigr)E_6\bigr)(\varphi_{0,m})=0.
\end{multline*}
For example, the equation \eqref{mde:K3} for the elliptic genus of $K3$-surface is equivalent to 
\begin{equation*}\label{HEphi01}
\bigl(
H^3-\frac 92E_2H^2+\frac 92(6E'_2-5E_4)H+27(E''_2-\frac 54E'_4)
\bigr)(\varphi_{0,1})=0.
\end{equation*}
The last equation yields a recurrence relation for the Fourier coefficients of the elliptic genus of a $K3$-surface \cite[Corollary 5.4]{AG2}.

In \cite[Proposition 3.4]{AG2} we proved that if a weak Jacobi form 
$\varphi_{k,m}=\sum_{n\ge 0, l} a(n,l)q^n\zeta^l$
of weight $k$ and index $n$ is a solution of a  Kaneko--Zagier type equation, then the $q^0$-part of  its  Fourier expansion
$$
q^0[\varphi_{k,m}]=\sum_{l=-m}^{m} a(0,l)\zeta^l
$$
must be of a rather special type.  Here, we provide simple necessary conditions on the $q^0$-part  of a weak Jacobi form satisfying a MDE of order 3. 
\begin{theorem}\label{MDE3k}
Let $\varphi_{k,m}\in J_{k,m}^w$ ($m\in \frac 12\Bbb N$) be a solution of a third-order MDE
$$
(H_k^{[3]}+\lambda_1E_4H_k+\lambda_2E_6)(\varphi_{k,m})=0.
$$
Then the following properties hold:

\noindent
1) the weight $k\geqslant -5$;
\smallskip

\noindent
2) the Fourier expansion of $\varphi_{k,m}$ contains at most three non-zero coefficients $a(0, l_1)$, $a(0, l_2)$, $a(0, l_3)$ with  $|l_i|\neq |l_j|$;
\smallskip

\noindent
3) suppose that there are exactly 3 such non-zero coefficients $a(0, l_1)$,
$a(0,l_2)$, $a(0,l_3)$, then the following is true:

a) $l_1^2+l_2^2+l_3^2=-\frac{2k+3}{2}m$, in particular, if $m$ is an integer, then each $l_i\in \Bbb Z$ and $m$ is even;

b) the weight $k$ is equal to $-2$, $-3$, $-4$, or $-5$;

c) the leading terms $q^0[\varphi_{k,m}]$ are as follows (up to a non-zero constant):
\begin{itemize}
    \item If $k=-2$:
    $$q^0[\varphi_{-2,m}]=(\zeta^{l_1}+\zeta^{-l_1})+a(\zeta^{l_2}+\zeta^{-l_2})-(a+1)(\zeta^{l_3}+\zeta^{-l_3}),\quad 
a\ne \frac{l_3^2-l_1^2}{l_2^2-l_3^2};
$$
    \item If $k=-4$:
    $$q^0[\varphi_{-4,m}]=(\zeta^{l_1}+\zeta^{-l_1})+\frac{l_3^2-l_1^2}{l_2^2-l_3^2}(\zeta^{l_2}+\zeta^{-l_2})+\frac{l_2^2-l_1^2}{l_2^2-l_3^2}(\zeta^{l_3}+\zeta^{-l_3});$$
We note that if $l_i=0$ then $\zeta^{l_i}+\zeta^{-l_i}=2$.
\smallskip

\item If $k=-3$:
    $$q^0[\varphi_{-3,m}]=(\zeta^{l_1}-\zeta^{-l_1})+a(\zeta^{l_2}-\zeta^{-l_2})-\frac{l_1+al_2}{l_3}(\zeta^{l_3}-\zeta^{-l_3}),\quad a\neq \frac{l_1(l_3^2-l_1^2)}{l_2(l_2^2-l_3^2)};$$
    \item If $k=-5$:
    $$q^0[\varphi_{-5,m}]=(\zeta^{l_1}-\zeta^{-l_1})+\frac{l_1(l_3^2-l_1^2)}{l_2(l_2^2-l_3^2)}(\zeta^{l_2}-\zeta^{-l_2})+\frac{l_1(l_1^2-l_2^2)}{l_3(l_2^2-l_3^2)}(\zeta^{l_3}-\zeta^{-l_3}).$$
\end{itemize}

\end{theorem}
\begin{proof}
Let us recall the action of the modular differential operator on a Fourier expansion.
If
$
\varphi_{k,m}(\tau, z)=\sum_{n\ge 0,\,l} a(n,l)q^n\zeta^l\in J_{k,m}^{w},
$ 
then
\begin{align}\label{Hk}
 H_k(\varphi_{k,m})(\tau, z) &=
 \sum_{n\ge 0,\,l} q^n\Bigl[
  \sum_l \Bigl(\frac 3m(4nm-l^2)-\frac{2k-1}2\Bigr)a(n,l)\zeta^l
 \\
  &\qquad +12(2k-1)\sum_{s=1}^n\sigma_1(s)a(n-s,l)\zeta^l \Bigr].
  \notag
\end{align}

Let us consider the leading term $q^0[\varphi_{k,m}]$, which contains $\zeta^l$. The action of the third-order modular differential operator on $\zeta^l$ is given by
$$
H_k^{[3]}:\zeta^l\mapsto -\left(\frac{3}{m}l^2+\frac{2k-1}{2}\right)\left(\frac{3}{m}l^2+\frac{2k+3}{2}\right)\left(\frac{3}{m}l^2+\frac{2k+7}{2}\right)\zeta^l.
$$
Suppose that $q^0[\varphi_{k,m}]$ contains two distinct terms, $\zeta^{l_1}$ and $\zeta^{l_2}$, with $|l_1|\neq |l_2|$.
Let 
$$
x_i=\frac{3}{m}l_i^2+\frac{2k-1}{2}.
$$ 
It follows that $x_1\neq x_2$. The third-order MDE yields the following system of equations:
\begin{equation*}
\left\{
\begin{aligned}
-x_1(x_1+2)(x_1+4)-\lambda_1 x_1+\lambda_2=0,\\
-x_2(x_2+2)(x_2+4)-\lambda_1 x_2+\lambda_2=0.\\
\end{aligned}
\right.
\end{equation*}
From this
\begin{equation*}
x_1^2+x_1x_2+x_2^2+6(x_1+x_2)+(\lambda_1+8)=0.
\end{equation*}
We can express the coefficients $\lambda_1$ and $\lambda_2$ in terms of $x_1$ and $x_2$ as 
\begin{equation}\label{lambda}
\begin{cases}
\lambda_1&=-(x_1^2+x_1x_2+x_2^2+6(x_1+x_2)+8),\\
\lambda_2&=-x_1x_2(x_1+x_2+6).
\end{cases}
\end{equation}

If the $q^0[\varphi_{k,m}]$-term also contains a third term, $\zeta^{l_3}$, with  $|l_3|\neq |l_1|$ and  $|l_3|\neq |l_2|$, then
$$x_1^2+x_1x_3+x_3^2+6(x_1+x_3)+(\lambda_1+8)=0$$
also holds.
Combining these two equations we get
$x_1+x_2+x_3+6=0$ or 
\begin{equation*}\label{l3}
l_1^2+l_2^2+l_3^2=-\frac{2k+3}{2}m.
\end{equation*}
This demonstrates that $q^0[\varphi_{k,m}]$ cannot contain more than three terms $\zeta^{l_i}$ with pairwise distinct absolute values of $l_i$. Furthermore, this equation implies that $k\le -2$.

We know that $\varphi(\tau, -z)=(-1)^k\varphi(\tau, z)$. This property determines the structure of $q^0[\varphi]$. Up to multiplication by a non-zero constant, $q^0$-part of the Fourier expansion  takes the following form:
$$
q^0[\varphi_{k,m}]=
\left\{
\begin{aligned}
\zeta^{l_1}+\zeta^{-l_1}+a(\zeta^{l_2}+\zeta^{-l_2})-(a+1)(\zeta^{l_3}+\zeta^{-l_3}) & \quad \text{for even } k,\\
\zeta^{l_1}-\zeta^{-l_1}+a(\zeta^{l_2}-\zeta^{-l_2})+b(\zeta^{l_3}-\zeta^{-l_3}) & \quad \text{for odd } k.\\
\end{aligned}
\right.
$$
For a negative weight $k$ the multiplicity of zero at 
$\zeta=1$ is at least $-k$. Furthermore, the multiplicity can be exactly $-k$ only if the negative weight is $k$.
The second derivative of the polynomial $q^0[\varphi_{k,m}](\zeta)$
at $\zeta =1$ is zero if and only if $a=\frac{l_1^2-l_3^2}{l_3^2-l_2^2}$.
Moreover, its fourth  derivative for this value of $a$ is  always non-zero. This 
proves part 3-a) for even $k$. A similar argument applies to the case of odd weights. We also note that $l_1\cdot l_2\cdot l_3\ne 0$ in this case.
\smallskip

We have proven statement 1) only if $q^0[\varphi_{k,m}]$ contains exactly three non-zero coefficients --- $a(0, l_1)$, $a(0,l_2)$, $a(0,l_3)$ --- with  $l_i\neq l_j$. 
We can apply the same argument regarding the multiplicity of zero at $\zeta=1$ to cases involving negative weights $k$, where there are fewer non-zero coefficients $a(0,l_i)$ with pairwise distinct absolute values of $l_i$.

If the negative weight $k$ is even, then there exists only one type (up to a constant) of $q^0$-term
$$
\varphi_{k,m}=(\zeta^{l_1}+\zeta^{-l_1})
-(\zeta^{l_2}+\zeta^{-l_2})+O(q).
$$
The second derivative does not vanish at $\zeta=1$.Therefore, in this case, there is only one possibility for the negative weight $k=-2$.

If the negative weight $k$ is odd, then there are two types (up to a constant) of $q^0$-terms
$$
\varphi_{k,m}=(\zeta^l-\zeta^{-l})+O(q)
$$
and 
$$
\varphi_{k,m}=(\zeta^{l_1}-\zeta^{-l_1})+c(\zeta^{l_2}-\zeta^{-l_2})
+O(q).
$$
In the first case, and in the second case when $c\ne -\frac{l_1}{l_2}$, the value
$\zeta=1$ is a simple zero and $k=-1$.
If $c= -\frac{l_1}{l_2}$, then the multiplicity of zero at $\zeta=1$ is 3, and $k=-3$.


If $q^0[\varphi_{k,m}]$ does not contain $\zeta^{l_i}$, then $\varphi_{k,m}$ is divisible by $\Delta(\tau)$. 
Let $q^0[\Delta^{-n}\varphi_{k,m}]\neq 0$. Then
$\Delta^{-n}\varphi_{k,m}$ satisfies the same third-order MDE as 
$\varphi_{k,m}$, because the operator $H_k$ is homogeneous with respect to $\Delta$ (see \eqref{H-eta}). Hence,  $q^0[\Delta^{-d}\varphi_{k,m}]$ contains no more than three $\zeta^{l_i}$ with different absolute values of $l_i$, and then its weight $k-12d$ is not less than $-5$. However, we have assumed that the weight $k$ is negative, which leads to a contradiction.
\end{proof}

\begin{example}
In \cite[Proposition 5.6]{AG2} we found third-order MDEs for  
$\vartheta^4$ and $\vartheta^5$. Therefore, similar MDEs exist for two Jacobi forms of weight $-4$ and $-5$:
$$
\varphi_{-2,1}^2(\tau, z)=\zeta^2-4\zeta+6-4\zeta^{-1}+\zeta^{-2}+O(q)=\vartheta(\tau, z)^4/\eta(\tau)^{12},
$$
$$
(H_{-4}^{[3]}-\tfrac{23}{4}E_4H_{-4}+\tfrac{81}4E_6)(\varphi_{-2,1}^2)=0,
$$
and 
$$
\varphi_{-2,1}^2\varphi_{-1,\frac 12}=
(\zeta^{\frac 52}-\zeta^{-\frac 52})-5(\zeta^{\frac 32}-\zeta^{-\frac 32})
+10(\zeta^{\frac 12}-\zeta^{-\frac 12})+O(q)=
\vartheta(\tau, z)^5/\eta(\tau)^{15},
$$
$$
(H_{-5}^{[3]}-\tfrac{236}{25}E_4H_{-5}+\tfrac{728}{25}E_6)(\varphi_{-2,1}^2\varphi_{-1,\frac 12})=0.
$$
\end{example}

\begin{example}\label{E42}
{\bf Third-order MDEs for holomorphic and nearly holomorphic Jacobi forms}.
The holomorphic Jacobi-Eisenstein series $E_{6,1}$ and $E_{4,2}$ satisfy third-order MDEs (see \cite[\S 5.4]{AG2}
$$
\bigl(H_6^{[3]}-\tfrac{173}{4}E_4H_6+154E_6\bigr)(E_{6,1})=0,\quad
\bigl(H_4^{[3]}-\tfrac{95}{4}E_4H_4+\tfrac{245}{4}E_6\bigr)(E_{4,2})=0.
$$
Another example is 
$$
\psi^{(-1)}_{0,1}(\tau,z)=q^{-1}+(\zeta^{\pm 2}+70)+q(70\zeta^{\pm 2}+32384 \zeta^{\pm 1}+131976)+O(q^2),
$$
which is an additional   generator of the ring $J_{*, */2}^{nh, \Bbb Z}$
of nearly holomorphic Jacobi forms with integral coefficients.
We have 
$$
\bigl(H_0^{[3]}-\tfrac{533}{4}E_4H_0+874E_6\bigr)(\psi^{(-1)}_{0,1})=0.
$$
\end{example}
\smallskip

\begin{theorem}\label{q0-Ord3}
Let $\varphi_{0,m}\in J_{0,m}^w$ satisfy a third-order MDE
\begin{equation*}\label{MDE30}
(H_0^{[3]}+\lambda_1E_4H_0+\lambda_2E_6)(\varphi_{0,m})=0.
\end{equation*}
The  $q^0$-part of its Fourier expansion contains 
exactly two non-zero coefficients, $a(0, l_1)$, $a(0, l_2)$,  with 
$0\le l_1<l_2\le m$. The following relations hold between the coefficients:
$$
\begin{cases}
(6l_1^2-m)a(0,l_1)&=-(6l_2-m)a(0,l_2) \quad{\rm if\ } l_1\ne 0,\\
ma(0,0)&=\,2(6l_2^2-m)a(0,l_2)\quad \,{\rm if\ } l_1=0.
\end{cases}
$$
Moreover, the coefficients $\lambda_1$ and $\lambda_2$ of the possible third-order MDE  are given in \eqref{lambda}.
\end{theorem}
\begin{proof} According to Theorem \ref{MDE3k}, the  $q^0$-part of the Fourier expansion  of  $\varphi_{0,m}$ contains at most two non-zero coefficients, $a(0, l_1)$, $a(0, l_2)$,  with $|l_1|\neq |l_2|$.

1) Let us assume that $q^0[\varphi_{0,m}]=0$. The operator $H_k$ is homogeneous with respect to $\Delta$.
After dividing by an appropriate power of the $\Delta$-function, we obtain a solution $\psi_{-12k,m}=\Delta^{-k}\varphi_{0,m}\in J_{-12k,m}^w$ of the third-order MDE with the same $\lambda_1$ 
and $\lambda_2$, such that $q^0[\psi_{-12k,m}]_{q^0}\ne 0$. We note that $\zeta=1$ is a zero of $q^0[\psi_{-12k,m}]$ of order at least $6k$.
According to Theorem \ref{MDE3k} 
$$
q^0[\psi_{-12k,m}]=b(\zeta^{l_2}-\zeta^{l_1}-\zeta^{-l_1}+\zeta^{-l_2})
$$ 
and the order of zero at $\zeta=1$ is equal to 2, which leads to a contradiction.
\smallskip

2) Let us assume that there exists only one  coefficient $a(0, l)\ne 0$ with 
$l\ge 0$ in $\varphi_{0,m}$.
We know (see \cite[Lemma 1.10]{Gr99}) that for any Jacobi form in  $J_{0,m}^w$,
\begin{equation}\label{wt0}
m\bigl(a(0,0)+2\sum_{0< l\le m} a(0,l)\bigr)=12\sum_{0< l\le m} l^2a(0,l).
\end{equation}
It follows that $l\ne 0$, $2ma(0,l)=12l^2a(0,l)$, and $m=6l^2$. 
We have an example of a Jacobi form with such a $q^0$-term:
$$
\varphi_{0,\frac 32}(\tau, 2lz)=\zeta^l+\zeta^{-l}+O(q)\in J_{0,6l^2}^w.
$$
For $m=6l^2$, we have
$$
\varphi_{0, 6l^2}(\tau, z)=a(0,l)\varphi_{0,\frac 32}(\tau, 2lz)+\Delta(\tau)\psi_{-12,6l^2}(\tau, z),
$$
where $\psi_{-12,6l^2}\in J_{-12,6l^2}^w$.
Since we know that $H_0(\varphi_{0,\frac 32})=0$, it follows that $\lambda_2=0$ in the MDE above, and
$$
(H_{-12}^{[3]}+\lambda_1E_4H_{-12})(\psi_{-12,6l^2})=0.
$$
This leads to a contradiction similar to Theorem \ref{MDE3k}.
\smallskip

3) We have proven that there are exactly two non-zero coefficients, 
$a(0, l_1)$, $a(0, l_2)$, in $q^0[\varphi_{0,m}]$ with 
$0\le l_1<l_2\le m$. The relations in the theorem follow from
\eqref{wt0}. 
\end{proof}
 
\begin{corollary}\label{EG-Ord3}
Let $M_{2m}$ be a Calabi--Yau variety in the strict sense of dimension $2m$ (see the introduction).
If the elliptic genus of $M_{2m}$ satisfies a third-order MDE, then 
$$
\chi(M_{2m},\tau, z)= 2(\zeta^{m}+\zeta^{-m})+\frac{2m(6m-1)}{(m-6l^2)}(\zeta^{l}+\zeta^{-l}) +O(q)
$$
where $0\le l<m$. In particular, $m-6l^2\ne 0$ and 
$\frac{2m(6m-1)}{(m-6l^2)}\in \mathbb{Z}$. In this case,
$$
e(M_{2m})=\frac{24m-24l^2}{m-6l^2}.
$$
\end{corollary}
\begin{proof} For a strict Calabi--Yau variety $M_{2m}$
of dimension  $2m$, the $q^0$-coefficient of the elliptic genus  
$\chi(M_{2m}, \tau, z)\in J_{0,m}^w$ starts with  $2\zeta^m$. 
The relation from Theorem \ref{q0-Ord3} in this case is given in the corollary. The formula also holds for $l=0$, since 
$\zeta^l+\zeta^{-l}=2$ when $l=0$.
\end{proof}

Since the elliptic genus is a Jacobi form with integer Fourier coefficients, Corollary 3.5 imposes a strong constraint on the integer $l$: $\frac{2m(6m-1)}{(m-6l^2)}$ must be an integer.
If  $2\le m\le 5$, this is true only for $l$ equals $0$ and  $1$.
If $n=6$ then there is only one variant $l=0$. In this case
$$
\varphi_{0,6}=2\zeta^6+140+2\zeta^{-6}+O(q).
$$ 
We note that 12 is the minimum dimension for which the elliptic genus is not determined by the Hirzebruch $\chi_y$-genus.
\smallskip

\noindent
{\bf Question on $CY_{12}$.} {\it Is there a Calabi--Yau variety  of dimension $12$  whose Hirzebruch  $\chi_y$-genus is equal to 
$2y^{12}+140y^{6}+2$?}
\smallskip

\begin{example} {\bf K3 and $\bold C \bold Y_{\bold 4}$}. In \eqref{mde:K3}, we present the third-order MDE satisfied by the elliptic genus of a $K3$ surface. In \cite{AG2}, we proved that the elliptic genus of a strict Calabi--Yau fourfold $CY_4$ satisfies a third-order MDE if and only if
$e(CY_4)=48$ or $e(CY_4)=-18$. In this cases, the elliptic genus is given by
$$
\chi(CY_4,\tau, z)=
\begin{cases}2\psi_{0,2}(\tau, z)&= 
2\zeta^{2}+44+2\zeta^{-2}+O(q),\\
\ \rho_{0,2}(\tau,z)&=
2\zeta^{2}-11\zeta-11\zeta^{-1}+2\zeta^{-2}+O(q).
\end{cases}
$$
The corresponding MDEs are:
$$
(H_0^{[3]}-\tfrac{263}4E_4+121E_6)(\psi_{0,2})=0,\quad
(H_0^{[3]}-\tfrac{335}4E_4-\tfrac{275}4E_6)(\rho_{0,2})=0.
$$
See \cite[\S 5.3, (4.4), (4.5)]{AG2} for details.
\end{example}

\begin{example}\label{phi02}
{\bf Other solutions of third-order MDEs in $J_{0,m}$.}
In \cite{AG2}, we proved that there exists exactly one additional function
$$
\varphi_{0,2}(\tau,z)=\zeta+4+\zeta^{-1}+O(q)\in J_{0,2}^w
$$ that satisfies a third-order MDE, namely:
\begin{equation*}\label{HkEphi02}
(H_0^{[3]}-\tfrac{47}{4}E_4H_0+
\tfrac{13}{4}E_6)(\varphi_{0,2})=0.
\end{equation*}
Another example appears in \cite{AG1}:
$$
\varphi_{0,4}(\tau,z)=\vartheta(\tau, 3z)/\vartheta(\tau, z)=
\zeta+1+\zeta^{-1}+O(q)\in J_{0,4}^w,
$$
which satisfies the third-order MDE:
$$
H_0^{[3]}(\varphi_{0,4})-\tfrac{107}{16}E_4H_0(\varphi_{0,4})
+\tfrac{23}{32}E_6\varphi_{0,4}=0.
$$
\end{example}

\section {MDEs of the minimal order 4 and 5 for Jacobi forms in $J_{0,3}^w$}

In our previous work \cite{AG2}, we demonstrated that a generic Jacobi form in the space	 $J_{0,2}^w$  
satisfies a MDE  of order $5$. We also identified five exceptional Jacobi forms within this space: three of these satisfy MDEs of order 3, one has an MDE of order 4, and the final one requires an MDE of order 6 (see \cite[Theorem 4.1]{AG2}).
Furthermore, we provided a method for constructing an MDE for any form $\psi\in J_{0,2}^w$, which allowed us to find MDEs for the elliptic genus of 4-dimensional  hyperk\"ahler 
varieties. The remainder of this paper will focus on studying MDEs for weak Jacobi forms of weight 0 and index 3.

\begin{theorem}\label{ind3wt0}
There are no Jacobi forms of weight 0 and index 3 that satisfy a third-order MDE.
\end{theorem}
\begin{proof}
In \cite{Gr99}, a natural basis for the free module of weak Jacobi forms with integral coefficients, $J_{0,m}^{w,\Bbb Z}$, was introduced. In particular, we have:
\begin{equation}\label{J03}
J_{0,3}^{w,\Bbb Z}=\langle\rho_{0,3}, \psi_{0,3}, \varphi_{0,3}\rangle_{\Bbb Z}
\end{equation}
where the basis elements are defined by their Fourier expansions as follows:
$$
\begin{aligned}
\varphi_{0,3}&=\left(\tfrac{\vartheta(\tau,2z)}{\vartheta(\tau,z)}\right)^2
=\zeta^{\pm 1}+2+
q(-2\zeta^{\pm 3}-2\zeta^{\pm 2}+2\zeta^{\pm 1}+4)+O(q^2),\\
\psi_{0,3}&=\varphi_{0,1}\varphi_{0,2}-14\varphi_{0,3}\\
{}&=\zeta^{\pm 2}+14+q(\zeta^{\pm 4}+40\zeta^{\pm 3}-76\zeta^{\pm 2}-168\zeta^{\pm 1}+406)+O(q^2),\\
\rho_{0,3}&=\varphi_{0,1}^2-30\psi_{0,3}-303\varphi_{0,3}\\
{}&=\zeta^{\pm 3}+34+q(-186\zeta^{\pm 3}+2430\zeta^{\pm 2}-8262\zeta^{\pm 1}+12036)+O(q^2) 
\end{aligned}
$$
where we use the shorthand notation $\zeta^{\pm n}=\zeta^{n}+\zeta^{-n}$.
Using this basis, any Jacobi form $\Phi_{0,3}\in J_{0,3}^w$ can be expressed as:
\begin{align}\label{Phi03}
\Phi_{0,3}&=a\rho_{0,3}+b\psi_{0,3}+c\varphi_{0,3}\\
{}&=a\zeta^{\pm 3}+b\zeta^{\pm 2}+c\zeta^{\pm 1}+2(17a+7b+c)\notag\\
{}&+q[b\zeta^{\pm 4}-2(93a-20b+c)\zeta^{\pm 3}+2(1215a-38b-c)\zeta^{\pm 2}\notag\\
{}&-2(4131a+84b-c)\zeta^{\pm 1}+2(6018a+203b+2c)]+O(q^2).\notag
\end{align}
Now, let's assume $\Phi_{0,3}$ satisfies a thord-order MDE, which has the form:
\begin{equation}\label{J63}
H_0^{[3]}(\Phi_{0,3})+\lambda_1E_4H_0(\Phi_{0,3})+\lambda_2 E_6\Phi_{0,3}=0.
\end{equation}
Each term in this equation belongs to the space $J_{6,3}^w$. This space contains only one Jacobi form  (up to a constant multiple) without $q^0$-term, namely  $\Psi_{6,3}^{(q)}=\Delta\varphi_{-2,1}^3$. Therefore, to obtain a third-order MDE for $\Phi_{0,3}$, we only need to control the $q^0$- and $q^1$-terms in equation \eqref{J63}.
The modular differential operator $H_k$ acts:
\begin{align}\label{H3k}
 H_k(\phi_{k,3})(\tau, z) &=
 \sum_{n\ge 0,\,l} q^n\Bigl[
  \sum_l \Bigl(12n-l^2-\frac{2k-1}2\Bigr)a(n,l)\zeta^l
 \\
  &\qquad +12(2k-1)\sum_{s=1}^n\sigma_1(s)a(n-s,l)\zeta^l \Bigr].
  \notag
\end{align}
For the operator $H_0^{[3]}$ acting on the space $J_{0,3}^w$, we find its action to be (see \eqref{Hk}):
$$
H_0^{[3]}:\zeta^l\mapsto -\left(l^2-\frac{1}{2}\right)\left(l^2+\frac{3}{2}\right)\left(l^2+\frac{7}{2}\right)\zeta^l.
$$
From this, the $q^0$-term of $H_0^{[3]}(\Phi_{0,3})$ is given by:
$$
q^0[H_0^{[3]}(\Phi_{0,3})]=
-\tfrac{17\cdot 21\cdot 25}{8}a\zeta^{\pm 3}
-\tfrac{8\cdot 11\cdot 15}{7}b\zeta^{\pm 2}
-\tfrac{45}{8}a \zeta^{\pm 2}
+\tfrac{21}{4}d
$$
where $d=17a+7b+c$. Using any computer algebra system, one can obtain the following expression for the $q^1$-term of $H_0^{[3]}(\Phi_{0,3})$:
$$
q^1[H_0^{[3]}(\Phi_{0,3})]=-\tfrac{1}{8}b \zeta^{\pm 4}+ \left(\tfrac{32589}{4}a-105 b+\tfrac{21}{4}c\right)\zeta^{\pm 3}+ 
$$
$$
+\left(\tfrac{2416635}{4}a-\tfrac{36393}{2}b-\tfrac{1989}{4}c\right)\zeta^{\pm 2}+\left(-\tfrac{27078705}{4}a-137655 b+\tfrac{3015}{4}c\right)\zeta^{\pm 1}+
$$
$$
+\tfrac{26788617}{2}a+\tfrac{1756839}{4}b+\tfrac{5001}{2}c.
$$

By Theorem \ref{q0-Ord3},  the term $q^0[\Phi_{0,3}]$ has exactly two non-zero coefficients, located at $\zeta^{l_1}$ and  $\zeta^{l_2}$ where $|l_1|\ne |l_2|$.
 This condition leads to precisely six possibilities where exactly two of the coefficients $a$, $b$, $c$, $17a+7b+c$ are 0.  In all these cases, the coefficients $\lambda_1$ and 
$\lambda_2$ of the MDE are uniquely defined by equation \eqref{lambda}, with $x_i=l_i^2-\frac{1}2$. 

When $a=0$ and $b=0$, the MDE coefficients are $\lambda_1=-\frac{33}{4},$ $\lambda_2=\frac{3}{2}.$ In this specific case, the  $q^1$-term of equation \eqref{J63} becomes
$$60c(\zeta^{\pm 3} -6\zeta^{\pm 2} +15\zeta-20)q.$$
As a result, the left side of equation \eqref{J63} equals $60c\Delta\varphi_{-2,1}^3.$

If $c=0$ and $17a+7b+c=0$, then $b=-\frac{17}{7}a$,  $\lambda_1=-\frac{777}{4},$ $\lambda_2=-\frac{1071}{2}.$ In this case, the $q^1$-term of  \eqref{J63} is equal to
$$1021020a(\zeta^{\pm 3} -6\zeta^{\pm 2} +15\zeta-20)q.$$
Thus, \eqref{J63} is $1021020a\Delta\varphi_{-2,1}^3$.

The remaining four cases are completely similar to these two cases.
\end{proof}

\begin{corollary}
Let $M_6$ be a compact complex  manifold of dimension 6 with vanishing first Chern class (over $\Bbb R$).
Then the elliptic genus of $M_6$ does not satisfy a third-order MDE.
\end{corollary}
\smallskip
 
In Section 2, we mentioned that no Jacobi form in $J_{0,3}^w$ satisfies a MDE of order 1 or 2. The following theorem provides a complete description of the minimal-order MDEs for forms within this space.
\begin{theorem}\label{Thm-MDE4}
There are exactly 10 weak Jacobi forms of weight 0 and index 3 that satisfy MDEs of order 4. This includes the basic Jacobi forms $\varphi_{0,3}$, $\psi_{0,2}$, and $\rho_{0,3}$ given in \eqref{J03}.
These ten forms satisfy the equations from \eqref{F1} to \eqref{F10}.
\end{theorem}
\begin{proof}
A fourth-order MDE can be written in the form:
\begin{equation}\label{MDE4}
H_0^{[4]}(\Phi_{0,3})+\lambda_1E_4H_0^{[2]}(\Phi_{0,3})+\lambda_2E_6H_0(\Phi_{0,3})+\lambda_3E_4^2\Phi_{0,3}=0.
\end{equation}
Note that the term $H^{[4]}_0(\Phi_{0,3})$ belongs to the space $J_{8,3}^w$. This space contains only one Jacobi form (up to a constant) that has no $q^0$-term, namely  $\Psi_{8,3}^{(q)}=\Delta\varphi_{-2,1}^2\varphi_{0,1}$. Therefore, just as in the case of the third-order MDEs discussed previously, to obtain a fourth-order MDE for  $\Phi_{0,3}$, only need to control the $q^0$- and $q^1$-terms in the equation \eqref{MDE4}.

For the space $J_{0,3}^w$, the modular differential operator $H_0^{[4]}$ acts on $\zeta^l$ as:
\begin{equation}\label{H04}
H_0^{[4]}:\zeta^l\mapsto -\left(l^2-\frac{1}{2}\right)\left(l^2+\frac{3}{2}\right)\left(l^2+\frac{7}{2}\right)\left(l^2+\frac{11}{2}\right)\zeta^l.
\end{equation}
This allows us to calculate the fourth iteration of the modular differential operator for a generic form
$\Phi_{0,3}=a\rho_{0,3}+b\psi_{0,3}+c\varphi_{0,3}\in J_{0,3}^w$. As demonstrated in the proof of Theorem, its $q^0$- and $q^1$-terms are
\begin{equation*}
q^0[H^{[4]}_0(\Phi_{0,3})]=\tfrac{258825}{16}a\zeta^{\pm 3}
+\tfrac{21945}{16}b\zeta^{\pm 2}+\tfrac{585}{16}c\zeta^{\pm 1}-\tfrac{231}{8}(17a+7b+c),
\end{equation*}
\begin{multline*}
q^1[H^{[4]}_0(\Phi_{0,3})]=
\tfrac{21945}{16}b\zeta^{\pm 4}-\tfrac{15}{8}(89403a-140b+7c)\zeta^{\pm 3}\\
+\tfrac{15}{8}(805545a-34426b-663c)\zeta^{\pm 2}
\\ -\tfrac{495}{8}(601749a+12236b-55c)\zeta^{\pm 1}+\tfrac{15}{8}(46439886a+1525181b+9038c).
\end{multline*}
By analysing the $q^0$-term of \eqref{MDE4}, we get the following system of linear equations:
$$
\left\{
\begin{aligned}
a\left(\tfrac{357}{4}\lambda_1-\tfrac{17}{2}\lambda_2+\lambda_3\right)&=-\tfrac{258825}{16}a,\\
b\left(\tfrac{77}{4}\lambda_1-\tfrac{7}{2}\lambda_2+\lambda_3\right)&=-\tfrac{21945}{16}b,\\
c\left(\tfrac{5}{4}\lambda_1-\tfrac{1}{2}\lambda_2+\lambda_3\right)&=-\ \tfrac{585}{16}c,\\
d\left(-\tfrac{3}{2}\lambda_1+\lambda_2+2\lambda_3\right)&=\ \ \ \tfrac{231}{8}d,\\
\end{aligned}
\right.
$$
with $d=17a+7b+c.$

If none of $a$, $b$, $c$, and $17a+7b+c$ are 0, the matrix of the linear system has a rank of 3, while the augmented matrix has a rank of 4. Therefore, {\it at least one and at most two of  $a$, $b$, $c$, and $17a+7b+c$ equal 0}. There are exactly 10 such cases, which we will now consider individually.
\smallskip

\noindent
1) If $a=0$, $b\neq 0$, $c\neq 0$, $17a+7b+c\neq 0$, then $\lambda_1=-\frac{197}{2}$, $\lambda_2=-146$, $\lambda_3=\frac{217}{16}$. 
An analysis of the $q^1$-term  with these parameters shows that the left side of  \eqref{MDE4}  is equal to $-2016(20b - c)\Psi_{8,3}^{(q)}$. Thus $c=20b$, and we get the first MDE of order 4:
\begin{equation}\label{F1}
(H_0^{[4]}-\tfrac{197}{2}E_4H_0^{[2]}
-146E_6H_0+\tfrac{217}{16}E_4^2)(\psi_{0,3}+20\varphi_{0,3})=0.
\end{equation}

\noindent
2) If $a\neq 0$, $b=0$, $c\neq 0$, $17a+7b+c\neq 0$, then $\lambda_1=-\frac{437}{2}$, $\lambda_2=-386$, $\lambda_3=\frac{697}{16}$.  In this case the $q^1$-term  equals $-4896(1215a-c)\Psi_{8,3}^{(q)}$.
This leads to the second equation for a fourth-order MDE:
\begin{equation}\label{F2}
(H_0^{[4]}-\tfrac{437}{2}E_4H_0^{[2]} 
-386E_6H_0+\tfrac{697}{16}E_4^2)(\rho_{0,3}+1215\varphi_{0,3})=0. 
\end{equation}

\noindent
3) If $a\neq0$, $b\neq 0$, $c=0$, $17a+7b+c\neq 0$, then 
$\lambda_1=-\tfrac{581}{2}$, $\lambda_2=-1106$, $\lambda_3=\tfrac{5593}{16}$. In this case the $q^1$-term  equals $-7200 (1377a+28 b)\Psi_{8,3}^{(q)}$.
We obtain the MDE:
\begin{equation}\label{F3}
(H_0^{[4]}-\tfrac{581}{2}E_4H_0^{[2]}
-1106E_6H_0+\tfrac{5593}{16}E_4^2)(28\rho_{0,3}-1377\psi_{0,3})=0.
\end{equation}

\noindent
4) If $a\neq0$, $b\neq 0$, $c\neq 0$, $17a+7b+c = 0$, then $c=-17a-7b$, 
$\lambda_1=-\tfrac{629}{2}$, $\lambda_2=-1442$, 
$\lambda_3=-\tfrac{5831}{16}$.  
In this case the $q^1$-term  equals $-2016 (5984 a + 189 b)\Psi_{8,3}^{(q)}.$
Then we get the equation for $189\rho_{0,3}-5984\psi_{0,3}+38675\varphi_{0,3}$:
\begin{equation}\label{F4}
(H_0^{[4]}-\tfrac{629}{2}E_4H_0^{[2]}
-1442E_6H_0-\tfrac{5831}{16}E_4^2)(189\rho_{0,3}-5984\psi_{0,3}+38675\varphi_{0,3})=0.
\end{equation}


\noindent
5) If $a=0$, $b=0$, $c\neq 0$, then $\lambda_2=2\lambda_1+51$, $\lambda_3=-\frac{\lambda_1}{4}-\frac{177}{16}$, and the $q^1$-term  equals $-12(29+2\lambda_1)\Psi_{8,3}^{(q)}$.
From this we obtain:
\begin{gather}\label{F5}
(H_0^{[4]}-\tfrac{29}{2}E_4H_0^{[2]}+22E_6H_0-\tfrac{119}{16}E_4^2)(\varphi_{0,3})=0.
\end{gather}

\noindent
6) If $a=0$, $b\neq 0$, $c= 0$, then $\lambda_2=5\lambda_1+\frac{693}{2}$, $\lambda_3=-\frac{7}{4}\lambda_1-\frac{2541}{16}$,  and the $q^1$-term  equals
$420(101+2\lambda_1)\Psi_{8,3}^{(q)}$.
Therefore,
\begin{gather}\label{F6}
(H_0^{[4]}-\tfrac{101}{2}E_4H_0^{[2]}+94E_6H_0-\tfrac{1127}{16}E_4^2)(\psi_{0,3})=0.
\end{gather}

\noindent
7) If $a\neq0$, $b=0$, $c= 0$, then  $\lambda_2=10\lambda_1+1799$, $\lambda_3=-\tfrac{17}{4}\lambda_1-\tfrac{14161}{16}$, and the $q^1$-term  equals
$27540(221+2\lambda_1)\Psi_{8,3}^{(q)}$.
This leads to the following MDE:
\begin{gather}\label{F7}
\bigl(H_0^{[4]}-\tfrac{221}{2}E_4H_0^{[2]}+694E_6H_0-\tfrac{6647}{16}E_4^2\bigr)(\rho_{0,3})=0.
\end{gather}

\noindent
8) If $a=0$, $b\neq 0$, $c\neq 0$, $17a+7b+c= 0$, we can set $b=1$, $c=-7$. In this case, we obtain $\lambda_2=6\lambda_1+445$, $\lambda_3=\frac{7}{4}\lambda_1+\frac{2975}{16}$, and the $q^1$-term equals 
$756(125+2\lambda_1)\Psi_{8,3}^{(q)}$. We get:
\begin{gather}\label{F8}
(H_0^{[4]}-\tfrac{125}{2}E_4H_0^{[2]}+70E_6H_0+\tfrac{1225}{16}E_4^2)(\psi_{0,3}-7\varphi_{0,3})=0. 
\end{gather}

\noindent
9) If $a\neq 0$, $b= 0$, $c\neq 0$, $17a+7b+c= 0$,  we can assume that $a=1$, $c=-17$.  In this case, we obtain 
 $\lambda_2=11\lambda_1+\frac{4035}{2}$, $\lambda_3=\frac{17}{4}\lambda_1+\frac{15555}{16}$, and the $q^1$-term  equals
$31416(245+2\lambda_1)\Psi_{8,3}^{(q)}$. 
We obtain:
\begin{gather}\label{F9}
(H_0^{[4]}-\tfrac{245}{2}E_4H_0^{[2]}+670E_6H_0+\tfrac{7225}{16}E_4^2)(\rho_{0,3}-17\varphi_{0,3})=0.
\end{gather}

\noindent
10) Finally, if $a\neq 0$, $b\neq 0$, $c= 0$, $17a+7b+c= 0$,  we can set $a=7$, $b=-17$. In this case, we obtain
$\lambda_2=14\lambda_1+2961$, $\lambda_3=\frac{119}{4}\lambda_1+\frac{143871}{16}$, and the $q^1$-term  equals
$249900(317+2\lambda_1)\Psi_{8,3}^{(q)}$. 
This leads to the following MDE:
\begin{gather}\label{F10}
(H_0^{[4]}-\tfrac{317}{2}E_4H_0^{[2]}
+742E_6H_0+\tfrac{68425}{16}E_4^2)(7\rho_{0,3}-17\psi_{0,3})=0.
\end{gather}
\end{proof}

The elliptic genus of a strict Calabi--Yau variety of dimension 6
is fully determined by its Hirzebruch $\chi_y$-genus, or equivalently, by
the  $q^0$-term  of the Fourier expansion of the corresponding Jacobi form of weight 0 and index 3. 
Moreover, all Fourier coefficients of this Jacobi form must be integers, and the coefficient of the highest power of $\zeta$ (which is $\zeta^3$) must be 2.
Only three weak Jacobi forms obtained above (in equations  
\eqref{F7}, \eqref{F9}, and \eqref{F2})  satisfy these criteria.  

\begin{corollary}\label{Cor-CYMDE4}
The smallest possible order for a MDE of the elliptic genera of 
six-dimensional strict Calabi--Yau varieties  is at least 4. There are only three possible
candidates for elliptic genera satisfying MDE of order 4: 
\begin{align*}
2\rho_{0,3}&=2\zeta^{\pm 3}+68+O(q),\\
\alpha_{0,3}&=2(\rho_{0,3}-17\varphi_{0,3}) =2\zeta^{\pm 3}
-34 \zeta^{\pm1}+O(q),\\
\beta_{0,3}&=2(\rho_{0,3}+1215\varphi_{0,3})=2\zeta^{\pm 3}+2430\zeta^{\pm 1}+4928+ O(q).
\end{align*}
\end{corollary}

\noindent
{\bf Question on $CY_6$.} {\it Are there Calabi--Yau six-folds with these specific Hirzebruch 
$\chi_y$-genera?} If so, they would have Euler characteristics of $72$, $-64$, and  $9792=2^6\cdot 3^2\cdot 17$, respectively.
\medskip

A generic weak Jacobi form of weight 0 and index 2  satisfies a MDE of order 5 (see \cite[Theorem 4.1]{AG2}). For example, the elliptic genera of 
a hyperk\"ahler variety deformation equivalent to $\hbox{Hilb}^2(K3)$ and a generalized Kummer fourfold both satisfy fifth-order MDEs (see \cite[(5.1) and (5.1)]{AG2}). However, fifth-order MDEs for forms in $J_{0,3}^w$ remain  exceptional.

\begin{theorem}\label{ThmMDE5}
There are only five Jacobi forms, denoted by  $\xi_{0,3}^{(i)}$ (see \eqref{F11}--\eqref{F15}), of weight 0 and index 3 that satisfy	 MDEs of order 5. Of these, only one Jacobi form  (see \eqref{F13}), given by
$$
2(\rho_{0,3}+867\varphi_{0,3})=2\zeta^{\pm 3}+1734 \zeta^{\pm 1}+3536+
O(q)
$$
could correspond to the elliptic genus of a strict Calabi--Yau manifold. Euler characteristics of this Calabi--Yau manifold should be 7008.
\end{theorem}
\begin{proof}
A MDE of order 5 could be written as:
\begin{equation}\label{MDE5}
(H_0^{[5]}+\lambda_1 E_4H_0^{[3]}+\lambda_2E_6H_0^{[2]}+\lambda_3E_4^2H_0+\lambda_4E_4E_6)(\Phi_{0,3})=0.
\end{equation}
All the terms in this equation belong to the space $J_{10,3}^w$.  The subspace of $J_{10,3}^w$
containing forms without a $q^0$-term, denoted by $J_{10,3}^w(q)$, is two-dimensional:
$$
J_{10,3}^w(q)=\Delta\langle E_4\varphi_{-2,1}^3, \varphi_{-2,1}\varphi_{0,2}\rangle
$$
where $\varphi_{0,2}$ is defined in Example \ref{phi02}.
We note that $\phi_{0,3}\in J_{0,3}^w$ is uniquely determined  by its $q^0$- and $q^1$-terms since $\dim{J_{10,3}^w}(q^2)=\dim{J_{-14,3}^w}=0$. Thus, to obtain a fifth-order MDE, we we only need to control the $q^0$- and $q^1$-terms on the left-hand side of \eqref{MDE5}.

Let us calculate  the $q^0$- and $q^1$-parts of the fifth iteration of the modular differential operator, $H_0^{[5]}(\Phi_{0,3})$, where $\Phi_{0,3}=a\rho_{0,3}+b\psi_{0,3}+c\varphi_{0,3}$ (see \eqref{H3k} and  \eqref{H04}). The result is:
\begin{gather}\label{5Ord}
H_0^{[5]}(\Phi_{0,3})=-\tfrac{8541225}{32}a\zeta^{\pm 3}-\tfrac{504735}{32}b\zeta^{\pm 2}-\tfrac{9945}{32}c\zeta^{\pm 1}+\\
+\tfrac{3465}{16}(17a+7b+c)
+q\left(-\tfrac{504735}{32}b\zeta^{\pm 4}+\tfrac{135}{16}(434503a-140b+7c)\zeta^{\pm 3}+\right.\notag\\
+\tfrac{15}{16}(805545a+228914b-663c)\zeta^{\pm 2}-\tfrac{45}{16}(46334673a+942172b-6575c)\zeta^{\pm 1}+\notag\\
+\left.\tfrac{135}{16}(46429414a+1520869b+8422c)\right)+O(q^2)
.\notag
\end{gather}
By analysing the $q^0$-term of \eqref{MDE5}, we obtain the following system of linear equations for the coefficients $(\lambda_1,\ldots, \lambda_4)$
$$
\left\{
\begin{aligned}
a\bigl(-\tfrac{8925}{8}\lambda_1+\tfrac{357}{4}\lambda_2-\tfrac{17}{2}\lambda_3+\lambda_4\bigr)&=\tfrac{8541225}{32}a,\\
b\bigl(-\tfrac{1155}{8}\lambda_1+\tfrac{77}{4}\lambda_2-\frac{7}{2}\lambda_3+\lambda_4\bigr)&=\ \tfrac{504735}{32}b,\\
c\left(-\tfrac{45}{8}\lambda_1	+\tfrac{5}{4}\lambda_2
-\tfrac{1}{2}\lambda_3+\lambda_4\right)&=\ \ \,\tfrac{9945}{32}c,\\
d\left(\tfrac{21}{4}\lambda_1-\tfrac{3}{2}\lambda_2+\lambda_3+2\lambda_4\right)&=-\tfrac{3465}{16}d,\\
\end{aligned}
\right.
$$
with $d=17a+7b+c$.
{\it If all the coefficients $a$, $b$, $c$, $d$ are non-zero}, the system of four linear equations has a unique solution:
$$
\lambda_1=-\tfrac{989}{2},\quad  \lambda_2=-3514,\quad \lambda_3=-\tfrac{53879}{16},\quad  \lambda_4=238.
$$ 
In this case, the $q^1$-term of the left-hand side  of  equation \eqref{MDE5} is:
\begin{multline*}
432\Delta((718845 a - 2660b +133 c)E_4\varphi_{-2,1}^3+\\+(2371500 a + 177800 b + 2396 c)\varphi_{-2,1}\varphi_{0,2}).
\end{multline*}
This expression is zero for 
$b=\tfrac{1065866}{7\cdot 3249}a$ and  $c=-\tfrac{14515025}{3249}a$.
This defines the first Jacobi form:
\begin{equation*}
\xi_{0,3}^{(1)}
=3249\rho_{0,3}+\tfrac{1065866}7\psi_{0,3}-14515025\varphi_{0,3},
\end{equation*}
which satisfies the fifth-order MDE:
\begin{equation}\label{F11}
\bigl(H_0^{[5]}-\tfrac{989}{2}E_4H_0^{[3]}-3514E_6H_0^{[2]}-\tfrac{53879}{16}E_4^2H_0+238E_4E_6\bigr)
(\xi_{0,3}^{(1)})=0. 
\end{equation}

If two of $a$, $b$, $c$, and $17a+7b+c$ equal 0, we have (up to a constant multiple) six Jacobi forms satisfying fourth-order MDEs (\eqref{F5}--\eqref{F10}). 
It remains to consider the four cases {\it when only one of $a$, $b$, $c$, $17a+7b+c$ is zero.}
\smallskip

Let $a=0$, $b\neq 0$, $c\neq 0$, $17a+7b+c\neq 0$. Then
$\lambda_2=\tfrac{4735}{4} + \tfrac{19}2\lambda_1$, $\lambda_3=\tfrac{31175}{16} + \tfrac{43}{4}\lambda_1$, 
$\lambda_4=-\tfrac{3115}{16} - \tfrac{7}{8}\lambda_1$. 
For these values, the $q^1$-term of the left-hand side of \eqref{MDE5} for $b\psi_{0,3}+c\varphi_{0,3}$ is:
\begin{multline*}
84 (20 b - c) (305 + 2 \lambda_1)
\Delta E_4\varphi_{-2,1}^3-\\
-
(5040 b (3551 + 38\lambda_1) +
  72 c (4415 + 38\lambda_1))\Delta\varphi_{-2,1}\varphi_{0,1}^2.
\end{multline*}
This provides two solutions:
$$c=20b,\  \lambda_1=-\tfrac{197}{2} \quad \text{ and } \quad c=-\tfrac{2618}{23}b,\  \lambda_1=-\tfrac{305}{2}.$$
The first one corresponds to the fourth-order MDE \eqref{F1}, the second gives a new fifth-order MDE for $\xi_{0,3}^{(2)}=23\psi_{0,3}-2618\varphi_{0,3}$:
\begin{equation}\label{F12}
\bigl(H_0^{[5]}-\tfrac{305}{2}E_4H_0^{[3]}
-265E_6H_0^{[2]}+\tfrac{4945}{16}E_4^2H_0
-\tfrac{245}{4}E_4E_6\bigr)(\xi_{0,3}^{(2)})=0.
\end{equation}
The calculations for the remaining cases are similar, and they lead to three more fifth-order MDEs.
\smallskip

If $b=0$, then $\xi_{0,3}^{(3)}=\rho_{0,3}+867\varphi_{0,3}$ and 
\begin{equation}\label{F13}
\bigl(H_0^{[5]}-\tfrac{425}{2}E_4H_0^{[3]}+
575E_6H_0^{[2]}+\tfrac{39745}{16}E_4^2H_0
-\tfrac{1445}{4}E_4E_6\bigl)(\xi_{0,3}^{(3)})=0.
\end{equation}

If $c=0$, then $\xi_{0,3}^{(4)}=154\rho_{0,3}+1173\psi_{0,3}$ and 
\begin{equation}\label{F14}
\bigl(H_0^{[5]}-\tfrac{497}{2}E_4H_0^{[3]}+
{791}E_6H_0^{[2]}+\tfrac{145873}{16}E_4^2H_0
-\tfrac{13685}{4}E_4E_6\bigl)(\xi_{0,3}^{(4)})=0.
\end{equation}

If $17a+7b+c=0$, then 
$\xi_{0,3}^{(5)}=189\rho_{0,3}+1564\psi_{0,3}-14161\varphi_{0,3}$
and 
\begin{equation}\label{F15}
\bigl(H_0^{[5]}-\tfrac{521}{2}E_4H_0^{[3]}+
 815E_6H_0^{[2]}+\tfrac{188545}{16}E_4^2H_0
 +\tfrac{14875}{4}E_4E_6\bigr)(\xi_{0,3}^{(5)})=0.
 \end{equation}
\end{proof}

\section{The generic case: MDEs of order 7}

\subsection{MDE for a generic Jacobi form}
In the previous section, we studied MDEs of minimal possible orders that can be satisfied by Jacobi forms of weight 0 and index 3. In this section, we turn to the generic case and prove the following theorem.

\begin{theorem}\label{TMDE7}
Let
$\Phi_{0,3}=a\rho_{0,3}+b\psi_{0,3}+c\varphi_{0,3}\in J_{0,3}^w$
(see \eqref{Phi03}). Then $\Phi_{0,3}$ satisfies a MDE of order 7 assuming that $(a,b,c)$ does not belong to a plane cubic curve $S(a,b,c)=0$ in $\Bbb P^2(\Bbb C)$ defined by equation
\begin{multline}
S(a,b,c)=
1196694415125 a^3 - 819233068500 a^2 b \\
-18637542000 a b^2
- 30184000 b^3 - 20559305385 a^2 c \notag\\
+1081838520 a b c  
+ 4504080 b^2 c + 16833519 a c^2 - 137844 b c^2 - 595 c^3.
\notag
\end{multline}
\end{theorem}
\begin{remark}
Using, for example, standard tools of SAGE, one can verify that $S(a,b,c)=0$ is a non-singular cubic curve. In affine coordinates with $c=1$, it can be written in the Weierstrass form 
$y^2=x^3+g_2x+g_3$ where
$$g_2=-2^{16}\cdot 3^{14}\cdot 5^4\cdot7^4\cdot17^4\cdot1514472811003,$$
$$g_3=2^{25}\cdot 3^{23}\cdot 5^6\cdot7^6\cdot17^6\cdot41\cdot107\cdot137\cdot163\cdot 401\cdot 677\cdot 1481.$$
The discriminant of this cubic curve equals 
$$
2^{58}\cdot 3^{42}\cdot 5^{16}\cdot7^{18}\cdot11^4\cdot13^2\cdot17^{12}\cdot19^2\cdot29^2\cdot37^2\cdot 21193\cdot 196687.
$$
\end{remark}

\subsection{Equations of order 6}
We first show that a MDE of order 6 remains exceptional and does not represent the general form of the MDE for Jacobi forms in $J_{0,3}^w$. We then extend the analysis to the case of MDEs of order 7. More precisely, we show that the solutions
$\Phi_{0,3}=a\rho_{0,3}+b\psi_{0,3}+c\varphi_{0,3}$
to sixth-order MDEs of the form
\begin{multline}\label{MDE6}
\bigl(H_0^{[6]}+\lambda_1E_4H_0^{[4]}+\lambda_2E_6H_0^{[3]}+\lambda_3E_4^2H_0^{[2]}+\\ 
+\lambda_4E_4E_6H_0
+(\lambda_5E_4^3+\lambda_6\Delta)\bigr)(\Phi_{0,3})=0
\end{multline}
lie on a divisor in the projective plane
$\Bbb P^2(\Bbb C)$ of coefficients $(a,b,c)$. 
\smallskip

All terms in \eqref{MDE6}  belong to $J_{12,3}^w$. In particular, we have:
\begin{equation}\label{MDE6q1}
J_{12,3}^w(q)=\Delta J_{0,3}^w=
\Delta\langle \rho_{0,3}, \psi_{0,3}, \varphi_{0,3}\rangle,
\   J_{12,3}^w(q^2)=\Delta^2 J_{-12,3}^w=\{0\}.
\end{equation}
This implies that, as in previous cases, it is sufficient to control only $q^0$- and $q^1$-parts in the left-hand side of \eqref{MDE6}. 
For that  we need explicit expressions for the $q^0$- and $q^1$-terms of  the sixth iteration of the differential operator
(see  \eqref{5Ord} for $H_0^{[5]}$)
\begin{gather}\label{Ord6}
H_0^{[6]}(\Phi_{0,3})=\\
=\tfrac{316025325}{64}a\,\zeta^{\pm 3}
+\tfrac{13627845}{64}b\,\zeta^{\pm 2}+\tfrac{208845}{64}c\,\zeta^{\pm 1}-\tfrac{65835}{32}(17a+7b+c)+\notag\\
+q\Bigl(\tfrac{13627845}{64}b\,\zeta^{\pm 4}
-\tfrac{135}{32}(20073719a-1820b+91c)\zeta^{\pm 3}+\notag\\
+\tfrac{1755}{32}(20655a+71442b-17c)\zeta^{\pm 2}-\tfrac{135}{32}(46334673a+942172b+10221c)\zeta^{\pm 1}+\notag\\
+\tfrac{135}{32}(232346038a+7686273b+53814c)\Bigr)+O(q^2).
\notag
\end{gather}
Requiring the vanishing of the $q^0$-term in \eqref{MDE6} leads, as before, to a system of four linear equations:
\begin{equation}\label{Sys4}
\left\{
\begin{aligned}
a\left(\tfrac{258825 }{16}\lambda_1-\tfrac{8925}{8}\lambda_2
+\tfrac{357}{4}\lambda_3-\tfrac{17}{2}\lambda_4+\lambda_5 \right)&=-\tfrac{316025325}{64}a,\\
b\left(\tfrac{21945}{16}\lambda_1-\tfrac{1155 }{8}\lambda_2
+\tfrac{77}{4}\lambda_3-\tfrac{7}{2}\lambda_4+\lambda_5 \right)&=
\ -\tfrac{13627845}{64}b,\\
c\left(\tfrac{585 }{16}\lambda_1-\tfrac{45 }{8}\lambda_2
+\tfrac{5}{4}\lambda_3-\tfrac{1}{2}\lambda_4+\lambda_5 \right)&=
\ -\tfrac{208845}{64}c,\\
d\left(-\tfrac{231 }{8}\lambda_1+\tfrac{21 }{4}\lambda_2-\frac{3}{2}\lambda_3+\lambda_4+2\lambda_5 \right)&=\ \ \ \tfrac{65835}{32}d,
\end{aligned}
\right.
\end{equation}
with $d=17a+7b+c$.  

Let us assume that  {\it all $a$, $b$, $c$, and $d$ are non-zero}. In this case, the system contains four linearly independent equations in five variables $(\lambda_1,\ldots,\lambda_5)$, and thus admits a one-dimensional space of solutions parametrised by $(\lambda_1\in \Bbb C)$. The other variables are then determined as:
 
$$\lambda_2=10332 +24\lambda_1, \quad\lambda_3=\tfrac{1308069}{16}+\tfrac{259}{2}\lambda_1,
$$
$$
\lambda_4=\tfrac{165123}{2}+112\lambda_1, \quad 
\lambda_5=-\tfrac{193851}{32}-\tfrac{119}{16}\lambda_1.
$$
We can calculate the $q^1$-part of the left-hand side of \eqref{MDE6}
for these values of $\lambda_2$, $\ldots$, $\lambda_5$, while keeping $\lambda_1$ and $\lambda_6$ free. This term belongs to the space  $J_{12,3}^w(q)$ (see \eqref{MDE6q1}), and hence must be of the form:
$$
q^1\hbox{-part of }\ \eqref{MDE6} = (\gamma_1\rho_{0,3}+\gamma_2\psi_{0,3}+\gamma_3\varphi_{0,3})\Delta\in J_{12,3}^w(q),
$$
where the coefficients $\gamma_i$ are linear expressions in $\lambda_1$ and $\lambda_6$:
$$ 
\gamma_1= 288 (79713 a+140 b-7 c)\lambda_1+a \lambda_6
+72 (129614511 a+380660 b-19033 c),
$$
$$
\gamma_2=1728(20655 a+1288 b-17 c)\lambda_1+b\lambda_6
+432 (43767945 a+2253216 b-36023 c),
$$
$$
\gamma_3=864(-378675 a-7700 b+229 c)\lambda_1+c\lambda_6
+2376 (-63858375 a-1298500 b+$$
$$+42337 c).
$$
The vanishing of the $q^1$-term requires the system of equations $\gamma_1=\gamma_2=\gamma_3=0$. These are three non-homogeneous linear equations in two variables $\lambda_1$ and $\lambda_6$. The determinant of the augmented matrix of this system  is a homogeneous polynomial $Q(a,b,c)$ of degree 3, explicitly given by:
\begin{multline*}
Q(a,b,c)= -770714002764288000 a^3-426427487456256000 a^2b 
\\
-100201778725662720 a^2 c-5502923661312000 a b^2-1058093403586560abc
\\
+72199471730688 a c^2+57940033536000 b^3+17489027727360 b^2c
\\-1196967518208 b c^2+8883302400 c^3.
\end{multline*}
\begin{remark}
Using standard computational tools such as SAGE \cite{SAGE}, one can verify that $Q(a,b,c)=0$ is a non-singular cubic curve. In affine coordinates with $c=1$, it can be transformed into the Weierstrass form 
$y^2=x^3+g_2x+g_3$, where
$$g_2=-2^{68}\cdot 3^{34}\cdot 5^4\cdot 103\cdot 28097\cdot 29633\cdot
33211\cdot 385471,$$
$$g_3=2^{103}\cdot 3^{52} \cdot 5^6 \cdot 383\cdot 15955591\cdot
380880396892777922981.$$
The discriminant of this cubic curve equals:
$$
2^{210}\cdot 3^{102}\cdot 5^{22}\cdot 7^{10}\cdot 11^4\cdot 13^2\cdot
17^4\cdot 19^2\cdot 23^2\cdot 29^2\cdot 37^2\cdot 107\cdot
31618141242898974947.
$$
\end{remark}

If $Q(a,b,c)=0$, then the $q^1$-part of \eqref{MDE6} vanishes for the following values of $\lambda_1$ and $\lambda_6$:
$$
\lambda_1=\tfrac{262607670 a^2-116095215 a b-216138 a c-380660 b^2+19033 b c}{4 (123930 a^2-71985 a b-102 a c-140 b^2+7 b c)},
$$
$$
\lambda_6= -\tfrac{3456 (101460809385 a^2 + 1366512105 a b - 21854980 b^2 - 
   72662964 a c + 1271249 b c - 8925 c^2)}{123930 a^2 - 71985 a b - 140 b^2 - 102 a c + 7 b c}.
$$
This gives the first one-parameter family of solutions of sixth-order MDEs associated with the cubic curve divisor $Q(a,b,c)=0$ in $\mathbb{P}^2(\mathbb{C})$.
\smallskip

To complete the analysis of MDEs of order 6, we now consider a second divisorial condition: namely, when {\it exactly one  of the coefficients $a$, $b$, $c$, $d=17a+7b+c$ is zero}.  (We note that the case of two zero coefficients gives us MDEs of order 4 in \S 4). As a representative example, we provide the formula for the case $a=0$; the other cases are similar.


If $a=0$, the system \eqref{Sys4} reduces to a system of three linearly independent equations in five variables. Consequently, the solution space is two-dimensional, and we can express the remaining variables in terms of $\lambda_1, \lambda_2\in \mathbb{C}$ as follows:
$$
\lambda_3=-\tfrac{197 }{2}\lambda_1+\tfrac{19 }{2}\lambda_2-\tfrac{262395}{16}, \quad 
\lambda_4=-146 \lambda_1+\tfrac{43 }{4}\lambda_2-\tfrac{57015}{2},$$
$$
\lambda_5=\tfrac{217 }{16}\lambda_1-\tfrac{7 }{8}\lambda_2+\tfrac{95445}{32}.
$$
Substituting these expressions into the left-hand side of \eqref{MDE6}, we obtain the $q^1$-term:
\begin{multline*}
-168 (20 b-c)(12 \lambda_1-\lambda_2+2175)\Delta\rho_{0,3}-\\
-(12096\lambda_1 (20 b-c)-864 \lambda_2 (119 b-2 c)-b \lambda_6+2160 (41160 b-1061 c))\Delta\psi_{0,3}+\\
(66528 \lambda_1 (20 b-c)-216 \lambda_2 (1540 b-51 c)+c \lambda_6+351600480 b-13224600 c)\Delta\varphi_{0,3}.
\end{multline*}
Thus, we have a system consisting of three linear equations in three variables. Solving the system yields the following homogeneous rational expressions:
$\lambda_1$, $\lambda_2$ and $\lambda_6$:
$$\lambda_1=-\tfrac{35(24596 b^2+8200 b c-97 c^2)}{4(1540 b^2+536 b c-5 c^2)},\qquad 
 \lambda_2=\tfrac{30 (25564 b^2+10160 b c-23 c^2)}{1540 b^2+536 b c-5 c^2},$$
$$\lambda_6=\tfrac{120960 (49588 b^2 - 5320 b c - 17 c^2)}{1540 b^2 + 536 b c - 5 c^2}.
$$
We observe that the common denominator $1540 b^2 + 536 b c - 5 c^2$ factors as the product of two linear forms: $110b-c$ and $14b+5c$. Note that none of all these cases corresponds to MDEs of order 4 or 5. Points $(0, b, 110b)$ and $(0, 5b, -14b)$ do not also belong to the cubic $S(a,b,c)$ from Theorem 5.1. Using any standard computer algebra system, one can check that the corresponding Jacobi forms indeed satisfy the following MDEs of order 7:
\begin{multline*}
\left(H_0^{[7]}+\lambda_1E_4H_0^{[5]}+(8 \lambda_1+2224)E_6H_0^{[4]}-\left(\tfrac{113 \lambda_1}{2}+\tfrac{131187}{16}\right)E_4^2H_0^{[3]}-\right.\\-\left(141 \lambda_1+\tfrac{153877}{4}\right)E_4E_6H_0^{[2]}+\\+\left(\left(\tfrac{2769\lambda_1}{16}+\tfrac{1648289}{32}\right)E_4^3-(158976 \lambda_1+48714048)\Delta\right)H_0-\\ \left.-\left(\tfrac{147 \lambda_1}{4}+\tfrac{204869}{16}\right)E_4^2E_6\right)(\psi_{0,3}+110\varphi_{0,3})=0
\end{multline*}
and
\begin{multline*}
\left(H_0^{[7]}+\lambda_1E_4H_0^{[5]}+(8 \lambda_1+2177)E_6H_0^{[4]}-\left(\tfrac{113 \lambda_1}{2}+\tfrac{140211}{16}\right)E_4^2H_0^{[3]}-\right.\\-\left(141 \lambda_1+\tfrac{156791}{4}\right)E_4E_6H_0^{[2]}+\\+\left(\left(\tfrac{2769\lambda_1}{16}+\tfrac{1673857}{32}\right)E_4^3-(1133568\lambda_1+175329792)\Delta\right)H_0-\\ \left.-\left(\tfrac{147 \lambda_1}{4}+\tfrac{51793}{16}\right)E_4^2E_6\right)(5\psi_{0,3}-14\varphi_{0,3})=0
\end{multline*}
for any $\lambda_1\in \Bbb C$.

We have proved the following result.

\begin{theorem}\label{TMDE6}
There exist  exactly  five one-parameter families  in $\Bbb P^2(\Bbb C)$
of Jacobi forms $\Phi_{0,3}=a\rho_{0,3}+b\psi_{0,3}+c\varphi_{0,3}$ that satisfy a MDE of order 6. The first family is defined by the plane cubic
$Q(a,b,c)=0$. The remaining four families  correspond to the cases when exactly one of the coefficients $a$, $b$, $c$, $d=17a+7b+c$ is equal to zero.
\end{theorem}

\begin{remark}\label{excep}
We note that any two of these five divisors intersect in a single common point, which represents a Jacobi form satisfying a  fourth-order MDE. Furthermore, all five divisors contain exactly one point corresponding to one of the five Jacobi forms that satisfy fifth-order MDEs.

For example, the cubic curve $Q(a, b, c)$ contains four special points related to the fourth-order MDEs \eqref{F1}--\eqref{F4} and one related to \eqref{F11} of fifth-order.   With the exception of the fifteen forms already listed, any other Jacobi form belonging one of these five divisors satisfies a sixth-order MDE, but no lower-order equation.
\end{remark}

As an illustration of sixth-order MDEs, we present an explicit example.
\begin{example}
The Jacobi form 
$$
\varphi_{0,1}(\tau,z)\cdot \varphi_{0,2}(\tau, z)=\zeta^{\pm 2}+14\zeta^{\pm 1}+42+O(q)\in J_{0,3}^w
$$
satisfies a MDE of order 6:
\begin{multline*}\label{MDE6b}
\bigl(H_0^{[6]}-\tfrac{1045}{8}E_4H_0^{[4]}+\tfrac{1215}{2}E_6H_0^{[3]}+\tfrac{17905}{8}E_4^2H_0^{[2]}+\\ 
-\tfrac{23245}{8}E_4E_6H_0
+(\tfrac{86975}{128}E_4^3-423360\Delta)\bigr)(\varphi_{0,1}\varphi_{0,2})=0.
\end{multline*}
This is the MDE of the minimal order for this Jacobi form.
\end{example}

\subsection{Equations of order 7.} We now prove Theorem \ref{TMDE7}.
A MDE of order 7 has  seven coefficients $\lambda_i$ and takes the form
\begin{multline*}
\bigl(H_0^{[7]}+\lambda_1E_4H_0^{[5]}+\lambda_2E_6H_0^{[4]}+\lambda_3E_4^2H_0^{[3]}+\\
+\lambda_4E_4E_6H_0^{[2]}+(\lambda_5E_4^3+\lambda_6\Delta)H_0+\lambda_7E_4^2E_6\bigr)(\Phi_{0,3})=0.
\end{multline*}
 All summands in this equation belong to the space $J_{14,3}^w$. Consider the subspace of Jacobi forms without $q^0$- and $q^1$-terms,
\begin{equation}\label{MDE7q1}
J_{14,3}^w(q)=\Delta J_{2,3}^w=
\Delta\langle E_8\varphi_{-2,1}^3,  E_4\varphi_{-2,1}\varphi_{0,2},
E_{4,2}\varphi_{-2,1}\rangle,
\end{equation}
where $E_{4,2}=1+O(q)$ is the Jacobi--Eisenstein series of weight 4 and index 2, and $\varphi_{0,2}=\zeta^{\pm 1}+4+O(q)$ (see Example \ref{E42} and Example \ref{phi02}). Moreover,
$$ 
 J_{14,3}^w(q^2)=\Delta^2 J_{-10,3}^w=\{0\}.
$$
Therefore, we can apply the same algorithm as in the case of sixth-order MDEs.
The seventh iteration of the modular differential operator $H_0$ (see \eqref{Ord6} for $H_0^{[6]}$) equals:
 $$
 H_0^{[7]}(\Phi_{0,3})=
 $$
$$=-\tfrac{12957038325}{128}a\zeta^{\pm 3}-\tfrac{422463195}{128}b\zeta^{\pm 2}-\tfrac{5221125 }{128}c\zeta^{\pm 1}+\tfrac{1514205}{64}(17a+7b+c)+
$$
$$+q\left(-\tfrac{422463195}{128}b\zeta^{\pm 4}+\tfrac{2295}{64} 
(58079379 a - 1820 b + 91 c)\zeta^{\pm 3}+\right.
$$
$$+\tfrac{945}{64} (268515 a + 4908942 b - 221 c)\zeta^{\pm 2}+\tfrac{135}{64} (46334673 a + 942172 b + 437193 c) \zeta^{\pm 1}+
$$
$$\left.+\tfrac{135}{64} (227769774 a + 5801929 b - 215378 c)\right)+O(q^2).
$$ 
We have previously considered all cases in which at least one of the parameters
$a$, $b$, $c$, and $d=17a+7b+c$ equals zero. Then the corresponding Jacobi form satisfies a MDE of order smaller than 7. Hence, we assume that all of these parameters are non-zero.

The condition that the $q^0$-term vanishes is equivalent to a system of four linear equations in six variables $\lambda_i$ ($i\ne 6$). By \eqref{MDE7q1}, the vanishing of the $q^1$-term corresponds to a system of three linear equations in all seven variables $\lambda_i$. Thus, we obtain a combined system of seven linear equations in seven variables $\lambda_i$. Due to the cumbersome nature of the coefficients, we omit the explicit form of these equations.
Due to the complexity and size of the coefficients involved, we omit the explicit form of these equations here. However, using any standard computer algebra system, one can verify that this system has full rank (i.e., rank 7) and admits a unique solution for the coefficients $\lambda_i$:
\begin{multline*}
\lambda_1=-\tfrac{7}{4}\big(369778574273625 a^3 - 187963583791500 a^2 b - 
       4562993358000 a b^2 \notag\\
       - 11957176000 b^3 - 5245637718285 a^2 c + 
       253681717320 a b c + 1679228880 b^2 c \notag\\
       + 4391088219 a c^2 - 
       48732204 b c^2 - 266815 c^3\big)S(a,b,c)^{-1},\notag
\end{multline*}
\begin{multline}
\lambda_2=-\tfrac{35}{2}\big(35661493570725 a^3 + 109204695027450 a^2 b + 
       1896415844400 a b^2 \notag\\-6291639200 b^3 + 1989223666875 a^2 c - 
       127739623500 a b c + 639507120 b^2 c \notag\\
       -1402963641 a c^2 - 
       11992518 b c^2 - 212687 c^3\big)S(a,b,c)^{-1},\notag
 \end{multline}
\begin{gather}
\lambda_3=\tfrac{35}{16} \big(14595124425747525 a^3 - 9861171481289700 a^2 b - 224915175500400 a b^2 \notag\\
     - 373390740800 b^3 - 
     312305033436345 a^2 c + 12022287897240 a b c + 
     71640241680 b^2 c \notag\\
     +250367049279 a c^2 - 2595651972 b c^2 - 
     2644163 c^3\big)S(a,b,c)^{-1},\notag
\end{gather}
\begin{gather}
\lambda_4=-\tfrac{105}{4} \big(3501448079028075 a^3 - 6130716608674350 a^2 b - 
       123043479913200 a b^2 \notag\\+ 62357125600 b^3 - 
       94902395883075 a^2 c + 8237307080100 a b c - 
       10542585200 b^2 c \notag\\+74688977913 a c^2 + 353983826 b c^2 + 862631 c^3\big)S(a,b,c)^{-1},\notag
\end{gather}
\begin{gather}
\lambda_5=-\tfrac{105}{64} \big(165083116990589325 a^3 - 187504003928033100 a^2 b \notag\\-
       3937911921913200 a b^2 - 1157755614400 b^3 - 
       3704692902817905 a^2 c \notag\\+247650717458760 a b c + 
       192415853840 b^2 c + 2957957635575 a c^2 - 6386394476 b c^2   \notag\\-
 17000459 c^3\big)S(a,b,c)^{-1}, \notag
\end{gather}
\begin{gather}
 \lambda_6=77760 \big(51509211337388475 a^3-658034947889550 a^2 b-71691022151955 a^2 c \notag\\-48977968639600 a b^2+3428753823460 a b c+26471918441 a c^2-290771527200 b^3 \notag\\+15965426960 b^2 c-32503310 b c^2-1941961 c^3\big)S(a,b,c)^{-1}, \notag
\end{gather}    
\begin{gather}
\lambda_7=\tfrac{1785}{32} \big(530295185155725 a^3 - 537927629327550 a^2 b - 11468201883600 a b^2\notag\\
- 6317511200 b^3 - 11306435690205 a^2 c + 707976760260 a b c + 1028478640 b^2 c \notag\\+ 9066923871 a c^2 - 33647614 b c^2 - 99127 c^3\big)S(a,b,c)^{-1}.\notag
\end{gather}
All coefficients $\lambda_i$ have the same denominator $S(a,b,c)$, given by the cubic polynomial $S(a,b,c)$ defined in the statement of Theorem \ref{TMDE7}. Moreover, the numerators for all $\lambda_i$ are homogeneous cubic polynomials in $(a,b,c)$.
This completes the proof of Theorem \ref{TMDE7}.

\begin{remark}
 The cubic curve $S(a,b,c)=0$ intersects each of the five divisors listed in Theorem \ref{TMDE6}. For instance, the intersection of the cubic curves $S(a,b,c)=0$ and $Q(a,b,c)=0$ consists of nine common points.  These nine points correspond to: four Jacobi forms satisfying the fourth-order MDEs \eqref{F1}–\eqref{F4}, one satisfying the fifth-order MDE \eqref{F11}, and four forms (having rather complicated coefficients $(a,b,c)$) satisfying sixth-order MDEs. 
 
If $S(a,b,c)=0$ and $(a,b,c)$ does not belong to any of the five divisors from
Theorem \ref{TMDE6}, the corresponding Jacobi form $a\rho_{0,3}+b\psi_{0,3}+c\varphi_{0,3}$ satisfies a MDE of order higher than 7. Numerical experiments show that in these cases the order of MDEs is 8, although theoretically it could be even higher. While our method could be used to study all points on the plane cubic, we do not consider this specific issue in the present paper.



\end{remark}

\begin{example}\label{KAOG}{\bf MDEs of the elliptic genera of hyperk\"ahler varieties of dimension 6.}
In \cite{AG2}, MDEs of order 5 were found for the elliptic genera of two types of hyperk\"ahler varieties of dimension 4. In dimension 6, three  types of hyperk\"ahler varieties are known: varieties $K_6$ of $\hbox{Hilb}^{[3]}(K3)$-type, varieties $A_6$ of $\hbox{Kum}_3(A)$-type, and  varieties of $\hbox{OG}_6$-type. Their Hodge numbers have been computed (see \cite{GS}, \cite{MRS}, and the review \cite{BD}), which yield explicit formulas for their Hirzebruch $\chi_y$-genus and elliptic genus. 
The corresponding elliptic genera are given by:
$$
EG(K_6)=
4(\rho_{0,3}+16\psi_{0,3}+127\varphi_{0,3})=
4\bigl(\zeta^{\pm 3}+16\zeta^{\pm 2}+127\zeta^{\pm 1}+512+O(q)\bigl),
$$
$$
EG(A_6)=4\bigl(\rho_{0,3}+2\psi_{0,3}+11\varphi_{0,3})=
4(\zeta^{\pm 3}+2\zeta^{\pm 2}+11\zeta^{\pm 1}+84+O(q)\bigr),
$$
$$
EG(OG_6)=4(\rho_{0,3}+6\psi_{0,3}+87\varphi_{0,3})=
4\bigr(\zeta^{\pm 3}+6\zeta^{\pm 2}+87\zeta^{\pm 1}+292+O(q)\bigl).
$$
It is straightforward to verify that these Jacobi forms do not belong to the exceptional families corresponding to MDEs of orders 4 and 5.Specifically, the points $(1,16,127)$, $(1,2,11)$, and $(1,6,87)$ do not lie on the plane cubic curves $Q(a,b,c)=0$ and $S(a,b,c)=0$. Thus, each of these elliptic genera satisfies a MDE of order 7. Substituting the corresponding values of $a$, $b$, and $c$ into the expressions for $\lambda_i$ computed previously yields the following seventh-order MDEs:

\begin{gather*}\label{eqHilb}
\left(H_0^{[7]}-\tfrac{15121249171}{1747\cdot21391}E_4H_0^{[5]}+\tfrac{693022919057}{2^3\cdot1747\cdot21391}E_6H_0^{[4]}+ \tfrac{16455397736281}{2^4\cdot1747\cdot21391}E_4^2H_0^{[3]}-\right.\\ 
-\tfrac{112040811677277}{2^4\cdot1747\cdot21391}E_4E_6H_0^{[2]}-\left(\tfrac{438449632667109}{2^5\cdot1747\cdot21391}E_4^3+\tfrac{4495401229430736}{1747\cdot21391}\Delta\right)H_0+\notag\\\left.+\tfrac{172173893177097}{2^6\cdot1747\cdot21391}E_4^2E_6\right)(EG(K_6))=0; \notag
\end{gather*}

\begin{gather*}\label{eqKum}
\left(H_0^{[7]}-\tfrac{15426984953}{2^2\cdot20776739}E_4H_0^{[5]}+\tfrac{284631287485}{2\cdot20776739}E_6H_0^{[4]}+\right. \tfrac{9301863314105}{2^4\cdot20776739}E_4^2H_0^{[3]}-\\-\tfrac{30760061368995}{2^2\cdot20776739}E_4E_6H_0^{[2]}-\left(\tfrac{793296354632355}{2^6\cdot20776739}E_4^3+\tfrac{111098404984708800}{20776739}\Delta\right)H_0+\notag\\\left.+\tfrac{36177330650265}{2^5\cdot20776739}E_4^2E_6\right)(EG(A_6))=0;
\notag
\end{gather*}

\begin{gather*}\label{eqOG}
\left(H_0^{[7]}-\tfrac{122813480461}{2^2\cdot11\cdot7214393}E_4H_0^{[5]}+\tfrac{432883583489}{2\cdot11\cdot7214393}E_6H_0^{[4]}+\right. \tfrac{36278381836753}{2^4\cdot11\cdot72143938}E_4^2H_0^{[3]}-\\-\tfrac{62444977311759}{2^2\cdot11\cdot7214393}E_4E_6H_0^{[2]}\notag-\left(\tfrac{1932747726669879}{2^6\cdot11\cdot7214393}E_4^3+\tfrac{46727205157384128}{11\cdot7214393}\Delta\right)H_0+\notag\\\left.+\tfrac{94397656019709}{2^5\cdot11\cdot7214393}E_4^2E_6\right)(EG(OG_6))=0. \notag
\end{gather*}
\end{example}

\section{Conclusion}

The results presented in this paper and in our previous works show that the order of MDEs for weak Jacobi forms of weight 0 can be quite high. One contributing factor is the absence of modular forms of weight 2 with respect to $SL_2(\Bbb Z)$. This restriction can be addressed by either introducing meromorphic coefficients in the MDEs (see \eqref{MDEdef}) or allowing the leading coefficient of the differential equation to be a modular form rather than a constant (such as 1).

In generic cases, these modifications can reduce the order of the equation by at least one. However, for certain exceptional  Jacobi forms, the improvement can be even more significant, since we  change the structure of the linear systems in our algorithm. Below, we provide two illustrative examples. A more comprehensive analysis would require a separate paper.

\begin{proposition} 
The Jacobi form $\varphi_{0,1}(\tau, z)$, one of the two generators and the elliptic genus of an Enriques surface, satisfies the following second-order differential equation:
$$
\left(H_2H_0-4\frac{Dj}{j}H_0-\frac 54 E_4\right)(\varphi_{0,1})=0
$$
where $j(\tau)=\frac{E_4^3(\tau)}{\Delta(\tau)}$ is the modular invariant. (Compare this equation with with the third-order MDE \eqref{mde:K3}.)
\end{proposition}
\begin{proof}
According  to \cite[(2.11) -- (2.12)]{AG1}, we have
$$
H_0(\varphi_{0,1})=-\frac 52 E_4\varphi_{-2,1},\quad 
H_2H_0(\varphi_{0,1})=10 E_6\varphi_{-2,1}+\frac 54 E_4\varphi_{0,1}.
$$
Substituting these into the expression, we obtain:
$$
\left(E_4H_2H_0+4E_6H_0-\frac 54 E_4^2\right)(\varphi_{0,1})=0
$$
or 
$$
\left(H_2H_0+4\frac{E_6}{E_4}H_0-\frac 54 E_4\right)(\varphi_{0,1})=0.
$$
We also recall the identity (see \cite{KZ}):
$$
\frac{E_6(\tau)}{E_4(\tau)}=-\frac 1{2\pi i}\frac {j'(\tau)}{j(\tau)}=
1-744q+159768q^2-36866976q^3+8507424792q^4+\dots,
$$
which implies that the differential equation can be rewritten in terms of the logarithmic derivative of the modular invariant $j(\tau)$.
\end{proof}

A generic Jacobi form of weight 0 and index 2 satisfies a MDE of order 5.
However, there exists a unique (up to a scalar multiple) Jacobi form
$$
\omega_{0,2}(\tau,z)=5\zeta^{\pm 2}-308\zeta^{\pm 1}-1122+O(q)\in J_{0,2}^w
$$
satisfying a MDE of order 6; see  \cite[Theorem 4.1]{AG2}.

In terms of meromorphic modular differential equations, one can show that this exceptional Jacobi form of weight 0 and index 2 satisfies the following third-order equation:
$$
\left(H_0^{[3]}-12\frac{Dj}{j}H_0^{[2]}+\frac{73}{4}E_4H_0-\frac{11}{4}E_6\right)(\omega_{0,2})=0.
$$




\noindent
\textbf{Acknowledgements.}
The first author acknowledges generous support and hospitality of the Max Planck Institute for Mathematics in Bonn. The second author was supported by the HSE University Basic Research Program.

The authors are grateful to the referee for their valuable comments, which helped to substantially improve the presentation of the paper.
\bibliographystyle{amsplain}

\begin{thebibliography}{20}

\bibitem{AG1} D. Adler, V. Gritsenko, {\it Elliptic genus and modular differential equations.} J. Geom. Phys., {\bf 181} (2022), 104662

\bibitem{AG2} D. Adler, V. Gritsenko, {\it Modular differential equations of the elliptic genus of Calabi–-Yau fourfolds.} J. Geom. Phys., {\bf 194} (2023), 104995.

\bibitem{BD}
P. Beri, O. Debarre,
{\it On the Hodge and Betti Numbers of Hyper-K\"ahler Manifolds.}
Milan J. Math. {\bf 90} (2022), 417--431.






\bibitem{EZ} M. Eichler, D. Zagier, 
{\it The theory of Jacobi forms.} Progress in Mathematics {\bf 55}.
Birkh\"auser, Boston, Mass., 1985.




\bibitem{GS} L. G\"ottsche, W. Soergel, {\it Perverse sheaves and the cohomology of Hilbert schemes of smooth algebraic surfaces.} Math. Ann. 296 (1993), 235–-246.

\bibitem{Gr99} V. Gritsenko, \textit{Elliptic genus of Calabi--Yau 
manifolds and Jacobi and Siegel modular forms.} St. Petersburg Math. J. 
\textbf{11} (1999), 781--804.


\bibitem{Gr20} V. Gritsenko,
\textit{Modified Elliptic Genus.} In Partition Functions and Automorphic Forms. (Eds. V. Gritsenko, V. Spiridonov.) Moscow Lectures vol. 5, 2020, pp. 87--119, Springer.



\bibitem{GN98} V. Gritsenko, V. Nikulin, 
{\it Automorphic forms and Lorentzian Kac--Moody algebras. Part II.}  
Internat. J. Math. \textbf{9} (1998), 201--275.



\bibitem{KK} M. Kaneko, M. Koike, 
\textit{On modular forms arising from a differential equations of hypergeometric type.} Ramanujan J. {\bf 7} (2003), 145--164.


\bibitem{KNS} M. Kaneko, K. Nagatomo, Y. Sakai,
\textit{The third order modular linear differential equations.}
Journal of Algebra {\bf 485} (2017), 332--352.

\bibitem{KZ} M. Kaneko, D. Zagier, 
\textit{Supersingular j-invariants, hypergeometric series, and Atkin’s 
orthogonal polynomials.} AMS/IP Stud. Adv. Math. {\bf 7} (1998), 97--126.

\bibitem{KYY} T. Kawai, Y. Yamada, S. K. Yang,
\textit{Elliptic Genera and $N=2$ Superconformal Field Theory.}
Nucl. Phys.  \textbf{B414} (1994), 191--212.

\bibitem{KS} K. Kawasetsu, Y. Sakai,
\textit{Modular linear differential equations of fourth order and minimal W-algebras,} Journal of Algebra
{\bf 506} (2018), 445--488.

\bibitem{MRS}
G. Mongardi, A. Rapagnetta, G.Sacca.
{\it  The Hodge diamond of O’Grady’s six- dimensional example.} Compos. Math. {\bf 154} (2018), 984--1013. 


\bibitem{K} T. Kiyuna, \textit{Kaneko--Zagier type equation for Jacobi forms of index 1.} Ramanujan J. {\bf 39} (2016), 347--362.

\bibitem{M} D. Mumford, 
{\it Tata lectures on theta I.} Progress in Mathem. {\bf 28}, Birkh\"auser, Boston, Mass., 1983.


\bibitem{SAGE} Sage Mathematics Software, The Sage Development Team (version 10.7). {\tt {https://www.sagemath.org/}}

\bibitem{T00} B. Totaro, \textit{Chern numbers for singular varieties 
and elliptic homology.} Ann. of Math. (2) \textbf{151} (2000), 
757--791.

\end{thebibliography}

\end{document}